\documentclass{article}

\usepackage[english]{babel}

\usepackage[letterpaper,top=2cm,bottom=2cm,left=3cm,right=3cm,marginparwidth=1.75cm]{geometry}

\usepackage{amsmath,amssymb,amsthm}
\usepackage{graphicx}
\usepackage{tikz}
\usepackage{circuitikz}
\usetikzlibrary{arrows,positioning}
\usepackage{caption}
\usepackage{subcaption}
\usepackage{xcolor}
\usepackage[colorlinks=true, allcolors=blue]{hyperref}
\DeclareMathOperator{\diag}{diag}
\DeclareMathOperator{\dom}{dom}
\DeclareMathOperator{\ran}{ran}
\DeclareMathOperator{\sgn}{sgn}

\DeclareMathOperator{\im}{Im}
\DeclareMathOperator{\Span}{span}
\DeclareMathOperator{\re}{Re}
\newcommand{\R}{\mathbb{R}}
\renewcommand{\i}{{\rm i}}
\newcommand{\C}{\mathbb{C}}

\newcommand{\Gc}{\mathcal{G}}
\newcommand{\Xc}{\mathcal{X}}

\newcommand{\eL}{\mathcal{L}}
\usepackage{calc}

\makeatletter
\newlength\@SizeOfCirc%
\newcommand{\CricArrowRight}[1]{%
    \setlength{\@SizeOfCirc}{\maxof{\widthof{#1}}{\heightof{#1}}}%
    \tikz [x=1.0ex,y=1.0ex,line width=.15ex, draw=black]%
        \draw [->,anchor=center]%
            node (0,0) {#1}%
            (0,1.2\@SizeOfCirc) arc (85:-240:1.2\@SizeOfCirc);%
}%
{\left(\begin{smallmatrix}}%            begin code
{\end{smallmatrix}\right)}%             end code

\newenvironment{smallbmatrix}%          environment name
{\left[\begin{smallmatrix}}%            begin code
{\end{smallmatrix}\right]}%             end code

\newtheorem{definition}{Definition}[section]

\newtheorem{proposition}[definition]{Proposition}
\newtheorem{lemma}[definition]{Lemma}
\newtheorem{corollary}[definition]{Corollary}
\newtheorem{remark}[definition]{Remark}

\theoremstyle{definition}

\title{Stability and passivity for a class of distributed port-Hamiltonian networks}

\author{Hannes Gernandt\thanks{Fraunhofer Research Institution for Energy Infrastructures and Geothermal Systems IEG, Gulbener Straße 23, 03046 Cottbus, Germany  (\texttt{hannes.gernandt@ieg.fraunhofer.de}).}
\and Dorothea Hinsen\thanks{ Institut f\"{u}r Mathematik, Technische Universität Berlin, Stra\ss e des 17.\ Juni 136, 10623 Berlin, Germany   (\texttt{hinsen@math.tu-berlin.de}).} }

\begin{document}
\maketitle

\begin{abstract}
We consider a class of infinite dimensional (distributed) pH systems which is invariant under Kirchhoff-type interconnections and prove exponential stability and a power balance equation for classical solutions. The results are illustrated for power networks that incorporate distributed transmission line models based on the telegraph  equations.
\end{abstract}

\section{Introduction}
 In the past years, the port-Hamiltonian (pH) modeling approach to systems theory has received a lot of attention. It is well known to guarantee passivity and stability but it has various other benefits such as its flexibility allowing for multiphysical modeling, availability of structure preserving model order reductions, as well as efficient numerical methods, see \cite{UnM22} for a recent overview.

In this note, we study pH models for partial differential equations (PDEs) or also called distributed parameter pH (dpH) systems in the literature, see e.g.\ the monograph \cite{JacZ12} for a concise introduction as well as the recent survey \cite{RasCSS20} which summarizes twenty years of developments in dpH systems. 

We consider the dpH system formulation from \cite{JacZ12,Vill07} which is the following class of abstract Cauchy problems on the Hilbert space $X=L^2([a,b],\C^n)$ of all measurable and square-integrable functions $x:[a,b]\rightarrow\C$ equipped with the scalar product $\langle x,y\rangle_{L^2}:=\int^b_a \overline{y(\xi)}^\top x(\xi)d \xi$
\begin{align}
\begin{split}
\tfrac{d}{dt}z(t)=AHz(t), \quad z(0)=z_0,\quad \text{ with } z_0 \in \C^n, \label{eq:ph_op}  
\end{split}
\end{align}
where $A:X\supset \dom A \rightarrow X$ is a  closed, densely defined linear operator  with domain $\dom A$ and $H$ is a bounded uniformly positive operator which are given by 
\begin{align}
\label{eq:A_JaZw_intro}
Ax(\xi):=P_1\frac{d}{d \xi}x(\xi)+P_0x(\xi),\quad  (Hx)(\xi):=\mathcal{H}(\xi)x(\xi),
\end{align}
where $P_1\in\mathbb{C}^{n\times n}$ is Hermitian and invertible, $P_0\in\mathbb{C}^{n\times n}$  skew-adjoint i.e. $P_0 = P^H_0$ with $P_0^H$ is the conjugate transpose of $P_0$ and the function $\mathcal H:[a,b]\rightarrow \mathbb C^{n\times n}$ is assumed to be measurable and satisfies $m I_n\leq \mathcal H(\xi)\leq M I_n$ for some constants $0<m<M$ for almost all $\xi\in[a,b]$. One key idea of pH systems is the concept of \emph{ports} which open the system for control, observation or interconnection. To include ports in \eqref{eq:ph_op} it is common to introduce the \emph{boundary flow} $f_{\partial}$ and \emph{boundary effort} $e_{\partial}$ which are given in terms of $z(t,\xi):=(z(t))(\xi)$ by 
\begin{align*}
%    \label{def_bef}
\begin{bmatrix}f_{\partial}(t)\\ e_{\partial}(t)\end{bmatrix}&:= \frac{1}{\sqrt{2}}\begin{bmatrix} P_1& -P_1\\ I_n & I_n
\end{bmatrix} \begin{bmatrix}
\mathcal{H}(b)z(t,b)\\ \mathcal{H}(a)z(t,a)
\end{bmatrix},
\end{align*}
see \cite{JacZ12,Vill07} and \cite[Definition 3.5]{LeGZM05}. Using these flow and effort variables, we obtain the following boundary control system, see e.g.\ \cite[Section~11.3]{JacZ12},
\begin{align}
\nonumber
\frac{\partial }{\partial t}z(t,\xi)&=P_1\frac{\partial}{\partial \xi}(\mathcal H(\xi)z(t,\xi))+P_0(\mathcal H(\xi)z(t,\xi)),\quad  W_{B,2}\begin{bmatrix}f_{\partial}(t)\\ e_{\partial}(t)\end{bmatrix}=0,\\
u(t)&=W_{B,1}\begin{bmatrix}f_{\partial}(t)\\ e_{\partial}(t)\end{bmatrix},\quad \quad y(t)=W_C\begin{bmatrix}f_{\partial}(t)\\ e_{\partial}(t)\end{bmatrix},\label{eq:ph_bc_int} 
\end{align}
with $W_{B,1},W_{B,2},W_C\in\C^{n\times 2n}$. In the study of the existence and solutions of boundary control systems, one considers the operator $A$ given by \eqref{eq:A_JaZw_intro} with domain 
\[
\dom A:=\left\{ x\in L^2([a,b],\C^n) ~\bigg|~ H x\in H^1([a,b],\C^n), \frac{1}{\sqrt{2}}\begin{smallbmatrix}
W_{B,1}\\ W_{B,2}
\end{smallbmatrix}\begin{bmatrix} P_1& -P_1\\ I_n & I_n
\end{bmatrix} \begin{bmatrix}
\mathcal{H}(b)x(b)\\ \mathcal{H}(a)x(a)
\end{bmatrix}=0\right\}.
\]
Here $H^1([a,b],\C^n)$ denotes the Sobolev space of functions whose weak derivative exists and is square integrable.  Note that there are characterizations for $A$ being the generator of a contraction or an exponentially stable semigroup using only the matrices $W_{B,1}$ and $W_{B,2}$, see e.g\ \cite[Theorem 7.2.4, Lemma 9.1.4]{JacZ12}. Furthermore, power-balance equations for twice continuously differentiable input functions $u$ were provided in \cite{JacZ12,Vill07}.

The main contribution of this work is to provide a general class of operators for which we can guarantee contractivity or exponential stability, as well as passivity for classical solutions and these properties are preserved for typical classes of Kirchhoff-type interconnections. The introduced class of distributed pH networks appears in several applications ranging from gas transport networks \cite{EggKLSMM18}, to interconnections of beams as well as electrical transmission line networks, see \cite{JacZ12}.

Besides our extension of the pH system class some recent extensions of the pH system class were done in the following directions: \cite{Skr21,KurZ15} study spatially higher dimensional systems, \cite{Aug20} consider higher order spatial derivatives, extension of Dirac structures to infinite dimensional spaces were studied in \cite{KurZVB10} and extensions of the system node approach to pH systems \cite{ReiPS23} and stabilization for a class of dissipative systems \cite{LeGRWLM2022}. 

In our class of distributed pH systems, compared to \eqref{eq:A_JaZw_intro}, the considered operators $A$ have the following block diagonal structure
\[\mathcal{H}=\begin{bmatrix}
\mathcal{H}_1&0\\0&\mathcal{H}_2
\end{bmatrix}
,\quad P_1=\begin{bmatrix}
0&-I_d\\-I_d&0
\end{bmatrix},\quad P_1\frac{\partial}{\partial \xi}=\begin{bmatrix}
0&-D\\-D&0
\end{bmatrix},\quad n=2d,
\]
and the derivative operator $D$ is a closed and densely defined operator in $L^2([a,b],\C^d)$. The domain of $D$ lies between the domain of the \emph{minimal realization} $D_{\min}$ of the spatial derivative and the \emph{maximal realization} $D_{\max}$ given by
\begin{subequations}
\begin{align}
\label{def:deriv}
&D_{\min}x=\tfrac{\partial}{\partial\xi}x,\quad \dom D_{\min}:=\{x\in H^1([a,b],\C^d) ~|~ x(a)=x(b)=0\},\\ %% yes 0 is correct
\label{def:deriv_max}
&D_{\max}x=\tfrac{\partial}{\partial\xi}x,\quad \dom D_{\max}:= H^1([a,b],\C^d).
\end{align}
\end{subequations}
Recall, that the \emph{adjoint} $A^*$ of a closed densely define linear operator $A$ in a Hilbert space $X$ is determined by the following domain and action
\[
\dom A^*:=\{y\in X~|~ \langle Ax,y\rangle_{X} =\langle x,w\rangle_{X}~ \text{for some $w \in X$ and all $x\in\dom A$} \},\quad A^*y:=w.
\]
Then it is an easy consequence that for the operators in \eqref{def:deriv} one has $D_{\min}^*=-D_{\max}$.

In the applications mentioned above, incorporating  dissipation requires a replacement of the skew-adjoint matrix $P_0$ by some bounded non-positive diagonal operator 
\[
P_0=\begin{bmatrix}
-G&0\\0&-R
\end{bmatrix},
\]
where $G,R:X\rightarrow X$ are bounded nonnegative linear operators between $X=L^2([a,b],\C^d)$. 

This motivates the following modifications and generalizations of the pH model \eqref{eq:ph_bc_int}:
\begin{itemize}
 \item[(M1)] Instead of $A$ given by \eqref{eq:ph_op}, we consider an arbitrary Hilbert space $X$ with scalar product $\langle \cdot,\cdot \rangle$ and define on $X \times X$ the block operator
    \[
    A=\begin{bmatrix}
    -G&D\\D&-R
    \end{bmatrix},
    \]
    where $D:X\supset\dom D\rightarrow X$ is closed, densely defined and skew-symmetric, i.e.\ $D\subseteq -D^*$. Furthermore, we assume that $G,R:X\rightarrow X$ are bounded and  fulfill $\langle Rx,x\rangle\geq 0$ and $\langle Gx, x \rangle \geq 0$ for all $x\in X$;
    \item[(M2)] For $H:X\times X\rightarrow X\times X$ bounded there exists $m>0$ such that $\langle H z,z\rangle\geq m\|z\|^2$ holds for all $z \in X \times X$;
    
    \item[(M3)] Instead of specifying the domain of $A$ in \eqref{eq:ph_op} in terms of the boundary flow and boundary effort $f_{\partial}$ and $e_{\partial}$, we assume that the symmetric operator $S=\begin{smallbmatrix}
    0&\i D\\\i D&0
    \end{smallbmatrix}$ has a so called \emph{boundary triplet} $\{\Xc,\Gamma_0,\Gamma_1\}$. 
   \end{itemize}

Assuming (M1)-(M3), we consider in Section \ref{sec:BCS} the following boundary control system 
\begin{align}
\label{eq:gen_ph}
\dot z(t)=\begin{bmatrix}
    -G&D^*\\D^*&-R
    \end{bmatrix}Hz(t),\quad \Gamma_0Hz(t)=u(t),\quad \Gamma_1Hz(t)=y(t),
\end{align}
where the mappings $\Gamma_0,\Gamma_1:\dom(A^*H)\rightarrow\Xc$ are given by a boundary triplet. Such systems were studied in a more general context in \cite{MalS06,MalS07} and we apply their results to derive a power-balance equation for classical solutions $z$ of \eqref{eq:gen_ph}
\begin{align}
\label{eq:PBE_int}
\langle H z(t),z(t)\rangle-\langle H z(0),z(0)\rangle\leq 2\int_0^t\re \langle u(\tau),y(\tau)\rangle d\tau-\left\langle\begin{bmatrix}R&0\\0&G
\end{bmatrix}H z(t),H z(t)\right\rangle.
\end{align}
The power balance equation \eqref{eq:PBE_int} implies that the considered boundary control systems are \emph{impedance passive} in the sense of \cite{Sta02}.  
Closely related to this are recent derivations of power balance equations for weak solutions \cite{EggKLSMM18} and mild solutions \cite{PhiSFMW21} for particular classes of pH systems.

Besides passivity, Section~\ref{sec:BCS} contains results on the stability of the strongly continuous semigroups generated by 
\[
\begin{bmatrix}
    -G&D^*\\D^*&-R
    \end{bmatrix}H|_{\ker( \Gamma_0H)}.
\]
It is shown that the generated semigroups are  exponentially stable if $R$ is \emph{uniformly positive}, i.e.\ there exists $r_0>0$ such that 
\[
\langle R x,x\rangle \geq r_0\|x\|^2,\quad \text{for all $x\in X$}
\]
and if either $D^*$ is surjective, or compact resolvent assumptions hold. These assumptions are fulfilled in many applications such as models for electrical or fluid transmission networks or interconnections of beams. 

For classical distributed pH systems \eqref{eq:ph_bc_int} from \cite{JacZ12,Vill07}, the operator $P_0$ is assumed to be skew-adjoint. Here we allow for $P_0$ to be dissipative and we use this dissipativity to show exponential stability. Furthermore, in the rich literature on exponential stability of pH systems \eqref{eq:ph_bc_int}, the stability is characterized in terms of boundary and interconnection conditions of the network, see \cite{LeGZM05,Vill07,VilZLGM09,AugJ14,RamLGMZ14,Aug20,TroW22,WauZ22}, which can be described using the matrices $W_{B,1}$, $W_{B,2}$ and $W_C$ in \eqref{eq:A_JaZw_intro} whereas in our approach we use the uniform positivity of the operator $R$. A related result allowing $P_0$ to be dissipative in \eqref{eq:A_JaZw_intro} was given in \cite[Theorem 6.33]{Vill07}, but compared to this, we consider a more general class of operators.

In Section \ref{sec:interconnect}, we study structure preserving interconnections of boundary control systems \eqref{eq:gen_ph}. Here we consider networks consisting of finitely many nodes and interconnecting edges such that for every edge there is an operator satisfying (M1)-(M3) and we consider then direct sums of these operators. It is shown that Kirchhoff-like interconnection conditions for this direct sum operator fulfill the assumptions (M1)-(M3). Hence, the results of Section \ref{sec:BCS} yield the power-balance equation \eqref{eq:PBE_int} as well as exponential stability of the evolution equation on the network. 

That Kirchhoff-type interconnections preserves impedance passivity of classical solutions of boundary nodes was shown in \cite{AalM13}, and contractivity and passivity for hyperbolic systems in \cite{ZwaLGM16}, see also \cite{BanB21}. Besides the boundary triplet approach which is used in this paper for interconnection of pH systems, a more general approach for pH systems \eqref{eq:ph_bc_int} using boundary systems is given in \cite{SAW20}.

Finally, in Section \ref{sec:TL} we apply our results to show stability and passivity of electrical transmission line networks which are modeled as distributed pH systems based on the telegraph equations. The considered networks consist of coupled transmission lines as edges and prosumers at each node providing either voltage or current inputs. More generally, lumped circuits containing distributed lines lead to partial-differential algebraic equations which were analyzed in \cite{ReiT05}.

\section{Stability and passivity of the system class}
\label{sec:BCS}
First, we recall some basic notations. Let $X$ be a Hilbert space with scalar product $\langle\cdot,\cdot\rangle_X$. Then  $X\times X$ equipped with the scalar product $\langle\begin{smallbmatrix}x_1\\ x_2\end{smallbmatrix},\begin{smallbmatrix}y_1\\ y_2\end{smallbmatrix}\rangle:=\langle x_1,y_1\rangle_X+\langle x_2,y_2\rangle_X$ is a Hilbert space. When it is clear from the context we will neglect the subscript and also write $\langle\cdot,\cdot\rangle$ instead of $\langle\cdot,\cdot\rangle_X$.
The set of bounded operators mapping a Hilbert space $X$ into itself will be denoted by $\eL(X)$.

Then we consider the following boundary control system
\begin{align}
    \label{def:bctrl}
\tfrac{d}{dt}z(t)=\mathfrak{A} Hz(t),\quad u(t)=\Gamma_0H z(t),\quad y(t)=\Gamma_1Hz(t),\quad Hz(0)=Hz_0
\end{align}
in the Hilbert space $X\times X$, mappings  $\Gamma_0,\Gamma_1:\dom\mathfrak{A}\rightarrow X$ and assume that the following conditions hold:
\begin{itemize}
    \item[\rm (A1)] Let $\hat D$ be closed densely defined and skew-symmetric and $R,G\in \eL(X)$ with $\langle Rx,x\rangle\geq 0$ and $\langle Gx,x\rangle\geq 0$ for all $x\in X$
    then $\mathfrak{A}$ is a closed and densely defined operator on $X\times X$ given by
    \[\mathfrak{A}:=\begin{bmatrix}-G&\hat D^*\\ \hat D^* & -R\end{bmatrix};\] 
    \item[\rm (A2)] $H\in\eL(X\times X)$ and there exist $m,M>0$ such that  $m\|x\|^2\leq \langle H x,x\rangle\leq M\|x\|^2$ holds for all $x\in X\times X$;
    \item[\rm (A3)] 
$A:=\mathfrak{A}|_{\ker( \Gamma_0)}=\begin{bmatrix}-G&-D^*\\ D & -R\end{bmatrix}$ for some closed and densely defined operator $D$ in $X$. 
\end{itemize}

In the following definition, we recall typical notions of stability of semigroups, see e.g.\ \cite{EngN00}. 
\begin{definition}
Let $(T(t))_{t\geq 0}$ be a strongly continuous semigroup in a Hilbert space $X$. Then this semigroup is said to be
\begin{itemize}
\item[\rm (a)] \emph{unitary} if $\|T(t)x\|=\|x\|$ holds for all $x\in X$ and all $t\geq 0$;
    \item[\rm (b)] \emph{bounded} if there exists $M>0$ such that $\|T(t)\|\leq M$ holds for all $t\geq 0$ and \emph{contractive} if the latter holds with $M=1$;
    \item[\rm (c)] \emph{strongly stable} if  $\lim\limits_{t\rightarrow\infty}\|T(t)x_0\|=0$ holds for all $x_0\in X$;
    \item[\rm (d)] \emph{exponentially stable} if there exist $\alpha,M>0$ such that $\|T(t)\|\leq Me^{-\alpha t}$ holds for all $t\geq 0$.
\end{itemize}
\end{definition}
Strongly continuous semigroups satisfy the following relationship
\[
\text{exponentially stable}\quad \Rightarrow\quad\text{strongly stable} \quad \Rightarrow\quad\text{contractive}  \quad \Leftarrow\quad\text{unitary}.
\]

In the following, we use the \emph{resolvent set} $\rho(A)$ which is the set of all $\lambda\in\C$ for which $A-\lambda$ maps $\dom A$ bijectively into $X$. It follows from the closed graph theorem that $(A-\lambda)^{-1}$ is a bounded operator. Furthermore, the \emph{spectrum} is defined as $\sigma(A):=\C\setminus\rho(A)$.

Below we recall well-known characterizations of unitarity due to Stone, see \cite[Theorem II.3.24]{EngN00},
contractivity due to Lumer and  Philips, see \cite[Theorem II.3.15]{EngN00}, strong stability due to Arendt, Batty, Lyubich, and  V\~{u}, see \cite[Theorem V.2.21]{EngN00} and exponential stability from Gearhart, Prüss, and Huang, see \cite[Theorem V.1.11]{EngN00}.
\begin{proposition}
\label{prop:many_stable}
Let $(T(t))_{t\geq 0}$ be a strongly continuous semigroup  in a Hilbert space $X$ with generator $A$ then the following holds:
\begin{itemize}
\item[\rm (a)] The semigroup is unitary if and only if $A$ is skew-adjoint.
\item[\rm (b)] The semigroup is contractive if and only if 
\begin{align}
\label{eq:dis}
\langle Ax,x\rangle+\langle x,Ax\rangle\leq 0,\quad \text{for all $x\in \dom A$,}
\end{align}
i.e.\ $A$ is \emph{dissipative} and $A-\lambda$ is surjective for some $\lambda>0$.

\item[\rm (c)] The semigroup is strongly stable if it is bounded, $A^*$ has no eigenvalues on $\i\R$ and $\sigma(A)\cap \i\R$ is at most countable.

\item[\rm (d)] The semigroup is exponentially stable if and only if
\[
\sigma(A)\subseteq\{z\in\C~|~\re z< 0\}\quad \text{and}  \quad\sup_{\beta\in\R} \|(\i\beta-A)^{-1}\|<\infty.
\]
\end{itemize}
\end{proposition}
\begin{remark}
\label{rem:stable}
\begin{itemize}
    \item[\rm (a)] Since $A$ given by (A3) is a bounded perturbation of a skew-adjoint operator, it generates a strongly continuous semigroup, see e.g.\ \cite[Corollary III.1.5]{EngN00}. 
    \item[\rm (b)] The assertions (c) and (d) of Proposition \ref{prop:many_stable} also hold in Banach spaces. If we replace \eqref{eq:dis} by 
\[
\|(\lambda-A)x\|\geq \lambda\|x\|\quad \text{for all $\lambda>0$ and all $x\in\dom A$},
\]
then Proposition \ref{prop:many_stable} (b) holds in Banach spaces as well. 
\item[\rm (c)] For Hilbert spaces the surjectivity condition in (b) can be replaced by the dissipativity of $A^*$, see \cite[Theorem 6.1.8]{JacZ12}. 
\end{itemize}
\end{remark}

A useful characterization of exponential stability of semigroups in Hilbert spaces is based on Lyapunov inequalities for their generators, see e.g.\ \cite[Theorem 8.1.3]{JacZ12}.
\begin{proposition}
\label{prop:lyap}
Let $(T(t))_{t\geq 0}$ be a  strongly continuous semigroup in a Hilbert space $X$. Then the semigroup is exponentially stable if and only if there exists $P\in \eL(X)$ with $\langle Px,x\rangle\geq 0$ for all $x\in X$ satisfying the following \emph{Lyapunov inequality} 
\begin{align}
\label{lyap_strict}
\langle Ax,Px\rangle+\langle Px,Ax\rangle \leq-\langle x,x\rangle,\quad \text{for all $x\in \dom A$.}
\end{align}
\end{proposition}

Using the positivity assumption (A2) on the Hamiltonian $H$ we immediately obtain the following stability characterization for our class of operators.
\begin{proposition}
\label{prop:lyap_RG}
Let $A$ be a closed densely defined operator in the Hilbert space $X$ given by (A3). Then the following holds:
\begin{itemize}
\item[\rm (a)] If $R=G=0$ then $A$ generates a unitary semigroup.
\item[\rm (b)] If $\langle Gx,x\rangle ,\langle Rx,x\rangle\geq 0$ holds for all $x\in X$  then $A$ generates a contraction semigroup.
\item[\rm (c)] If there exist some $r_0,g_0>0$ such that $\langle Rx,x\rangle\geq r_0$ and $\langle Gx,x\rangle\geq g_0$ holds for all $x\in X$ then $A$ generates an exponentially stable semigroup.
\end{itemize}
\end{proposition}
Note that operators $G$ and $R$ fulfilling the condition in Proposition \ref{prop:lyap_RG} (c) are also called \emph{uniformly positive}. In the following, we investigate the exponential stability if $G$ is not uniformly positive and using the additional assumption 
\begin{itemize}
\item[(A4)] There exists $r_0>0$ such that $\langle R x,x\rangle \geq r_0\|x\|^2$ holds for all $x\in X$.
\end{itemize}
This assumption appears for instance in the context of beam equations and gas transport networks \cite{JacZ12,EggK18}. Thanks to the Lyapunov inequality \eqref{lyap_strict}, this also implies the exponential stability of the semigroup generated by $A$ with $G\geq 0$. Furthermore, similar stability results can be derived for the case where $R\geq 0$ holds and $G$ is uniformly positive. 

As a preliminary result, we show that the above assumption allows us to exclude eigenvalues on the imaginary axis.
\begin{lemma}
\label{lem:spec}
Let $D$ be a closed densely defined operator in a Hilbert space $X$ and consider $A=\begin{smallbmatrix}-G&-D^*\\D&-R\end{smallbmatrix}$ with $R,G\in \eL(X)$ satisfying (A2) and (A4) for some $r_0>0$. Then all eigenvalues $\lambda=r+\i \omega\in\sigma(A)$ with $\omega\neq 0$ satisfy
\begin{align}
\label{eq:eval_estim}
\re \lambda\leq \frac{r_0}{2}.
\end{align}
Furthermore, it holds that
\begin{align}
\label{ker_A}
\ker A=\ker A^*=\ker \begin{bmatrix}
G \\ D
\end{bmatrix}\times\{0\}.
\end{align}
\end{lemma}
\begin{proof}
Let $\lambda=r+\i\omega$ for some $r,\omega\in\R$ with $\omega\neq 0$ be an eigenvalue of $A$, then there exists $x=(x_1,x_2)\in X\times X$ satisfying $Ax=(r+\i\omega)x$.
This equation can be rewritten as
\begin{align}
\label{eq:A_eigenvalue}
Dx_1-Rx_2=(r+\i\omega)x_2,\quad -Gx_1-D^* x_2=(r+\i\omega)x_1.
\end{align}
Multiplying the first and the second equation with $x_2$ and $x_1$, respectively, implies
\begin{align}
\label{eq:first_sec}
\langle Dx_1,x_2\rangle-\langle Rx_2,x_2\rangle=(r+\i\omega)\|x_2\|^2,\quad -\langle Gx_1,x_1\rangle-\langle D^* x_2,x_1\rangle=(r+\i\omega)\|x_1\|^2.
\end{align}
The second equation implies 
\[
(r-\i\omega)\|x_1\|^2=-\langle Gx_1,x_1\rangle+\overline{\langle -D^* x_2,x_1\rangle}=-\langle Gx_1,x_1\rangle-\overline{\langle x_2,Dx_1\rangle}=-\langle Gx_1,x_1\rangle-\langle Dx_1,x_2\rangle
\]
and combining this with the first equation in \eqref{eq:first_sec} yields
\begin{align}
    \label{eq:real_im}
(r-\i\omega)\|x_1\|^2+\langle Gx_1,x_1\rangle=-(r+\i\omega)\|x_2\|^2-\langle Rx_2,x_2\rangle.
%-(r-\i\omega)\|x_1\|^2-\langle Rx_2,x_2\rangle =(r+\i\omega)\|x_2\|^2.
\end{align}
Since $\omega\neq 0$, a comparison of the imaginary parts in \eqref{eq:real_im} leads to
\[
\|x_1\|=\|x_2\|.
\]
Hence $x_2\neq 0$ because otherwise $x_1=0$ and therefore $x=(x_1,x_2)=0$ would not be an eigenvector. Therefore, (A4) and \eqref{eq:real_im} imply
\[
-r_0\|x_2\|^2\geq -\langle Rx_2,x_2\rangle\geq -\langle Rx_2,x_2\rangle-\langle Gx_1,x_1\rangle =2r\|x_2\|^2
\]
which leads to \eqref{eq:eval_estim}. 

We continue with the proof of \eqref{ker_A}.  Clearly, $\supseteq$ holds. Conversely, let $(x_1,x_2)\in\ker A$  and assume that $x_2\neq 0$ (otherwise it would be in the desired set). Then, rewriting the left hand side equation in \eqref{eq:A_eigenvalue} for $r=\omega=0$ implies 
\[
x_2=R^{-1}Dx_1\quad \Rightarrow\quad  (-G-D^*R^{-1}D)x_1=0\quad \Rightarrow\quad  x_1^*(-G-D^*R^{-1}D)x_1=0 .
\]
The nonnegativity of $G$ and $D^*R^{-1}D$ yields $x_1\in\ker D^*R^{-1}D$ which implies  $x_1\in\ker R^{-\tfrac{1}{2}}D=\ker D$, but then $x_2=0$ holds which proves the claim. Applying \eqref{ker_A} to $A^*$ we obtain
\[
\ker A^*=\ker \begin{bmatrix}
G \\ -D
\end{bmatrix}\times\{0\}=\ker \begin{bmatrix}
G \\ D
\end{bmatrix}\times\{0\}=\ker A.
\]
\end{proof}

The following lemma gives a criterion for the invertibility of block matrices having unbounded off diagonal entries. In the case of bounded entries this result was obtained in \cite[Section A.4]{Sta05}. 
\begin{lemma}
\label{lem:2x2_inv}
Let $D$ be closed and densely defined in a Hilbert space $X$. We consider for $\alpha,\beta\in\C$ the block operator matrix of the form 
\[
A_{\alpha,\beta}=\begin{bmatrix}
\alpha I_X & -D^*\\ D & \beta I_X
\end{bmatrix}.
\]
If $\alpha\beta\notin(-\infty,0]$ then $A_{\alpha,\beta}$ is boundedly invertible and its inverse is given by 
\begin{align}
\label{eq:inv_A}
A_{\alpha,\beta}^{-1}=\begin{bmatrix}
\beta I_X & D^*\\ -D & \alpha I_X
\end{bmatrix}\begin{bmatrix}
(\alpha\beta+D^*D)^{-1}&0\\ 0&(\alpha\beta+DD^*)^{-1}
\end{bmatrix}.
\end{align}
\end{lemma}
\begin{proof}
By \cite[Theorem V.3.24]{Kat95}, $D^*D$ and $DD^*$ are self-adjoint and nonnegative and  we have $\sigma(D^*D),\sigma(DD^*)\subseteq[0,\infty)$. Hence if $\alpha\beta\notin(-\infty,0]$ we have, by definition of the spectrum, invertibility of $\alpha\beta+D^*D$ and $\alpha\beta+DD^*$. We first conclude the injectivity of $A_{\alpha,\beta}$. Assume that there exists $x=(x_1,x_2)$ with $A_{\alpha,\beta}x=0$. Then 
\[
\alpha x_1-D^*x_2=0,\quad Dx_1+\beta x_2=0.
\]
Since $\alpha\beta\neq 0$ we have $\alpha\neq 0$ and hence combining the two equations gives
\[
\alpha^{-1}DD^*x_2+\beta x_2=0.
\]
If $\alpha\beta\notin(-\infty,0]$ then $x_2=0$ since $\alpha\beta+DD^*$ is invertible and therefore $x_1=0$. Hence $A_{\alpha,\beta}$ is injective.

On the other hand a right inverse of $A_{\alpha,\beta}$ is given by 
\[
A_r^{-1}:X\times X\rightarrow \dom D\times \dom D^*,\quad A_{r}^{-1}:=\begin{bmatrix}
\beta I_X & D^*\\ -D & \alpha I_X
\end{bmatrix}\begin{bmatrix}
(\alpha\beta+D^*D)^{-1}&0\\ 0&(\alpha\beta+DD^*)^{-1}
\end{bmatrix}.
\]
Hence, $A_{\alpha,\beta}$ is surjective and, since it is also injective,  $A_{\alpha,\beta}^{-1}$ exists and is bounded. Moreover, using the equation $A_{\alpha,\beta}A_{r}^{-1}=I=A_{\alpha,\beta}A_{\alpha,\beta}^{-1}$ and multiplying it with $A_{\alpha,\beta}^{-1}$ from the left we see that $A_r^{-1}=A_{\alpha,\beta}^{-1}$ holds which proves \eqref{eq:inv_A}. 
\end{proof}

\begin{lemma}
\label{lem:stabwithH}
Let $D:X\supseteq \dom D\rightarrow X$ be closed and densely defined, let $R,G\in\eL(X)$ nonnegative and $H\in \eL(X\times X)$ uniformly positive and consider $A=\begin{bmatrix}-G&-D^*\\D&-R\end{bmatrix}$. Then $AH$ generates a strongly continuous semigroup. Furthermore, $A$ generates a contraction (exponentially stable) semigroup if and only if $AH$ generates a contraction (exponentially stable) semigroup. 
\end{lemma}
\begin{proof}
Decompose
\[
AH=\begin{bmatrix}0&-D^*\\D&0\end{bmatrix}H-\begin{bmatrix}G&0\\0&R
\end{bmatrix}H,
\]
where the first summand is skew-adjoint in the weighted space $\langle H\cdot,\cdot\rangle$ and therefore generates a strongly continuous semigroup. Furthermore, $AH$ is a bounded perturbation of this generator which implies that $AH$ generates a semigroup as well, see Remark~\ref{rem:stable} (a).

The left-hand side of the Lyapunov inequality \eqref{lyap_strict} can be rewritten as 
\[
\langle Az,P z\rangle+\langle Pz,A z\rangle=\langle AHH^{-1}z,H^{-1}HPHH^{-1} z\rangle+\langle H^{-1}HPHH^{-1}z,AHH^{-1} z\rangle.
\]
Hence we consider the weighted product $\langle v,w\rangle_{H-}:=\langle H^{-1}v,w \rangle$ for all $v,w\in X\times X$. Moreover, we have 
\[
-\langle z,z \rangle=-\langle H H^{-1}z,HH^{-1}z \rangle\leq -\|H^{-1}\|^{-2}\langle H^{-1}z,H^{-1}z\rangle\leq -\|H^{-1}\|^{-3}\langle H^{-1}z,H^{-1}H^{-1}z\rangle.
\]
This implies for $\hat P:= \|H^{-1}\|^3HPH$ and $w=H^{-1}z$ for $z\in\dom A$ that
\[
\langle AHw,\hat Pw \rangle_{H-}+\langle \hat Pw,AHw \rangle_{H-}\leq -\langle w,w\rangle_{H-}.
\]
This proves that $AH$ generates an exponentially stable semigroup. The converse can be proven by repeating the above steps for $AH$ and with $H^{-1}$ instead of $H$. Finally, that $A$ generates a contraction semigroup if and only if $AH$ has this property follows from Proposition \ref{prop:many_stable} (b) and Remark \ref{rem:stable} (c). 
\end{proof}
Note that there are examples of semigroup generator $A$ and bounded operators $H$ such that $AH$ does not generate a semigroup \cite[Section 6]{ZwaLGMV09}.

The following proposition is one of the main results of this section.
\begin{proposition}
\label{prop:exp_stable}
Let $D$ be a closed densely defined operator in a Hilbert space $X$, let $R,G\in \eL(X)$ be as in (A1) and assume that (A4) holds. Further, let $H\in\eL(X\times X)$ be uniformly positive then for  $A=\begin{smallbmatrix}
-G&-D^*\\D&-R
\end{smallbmatrix}$ the following holds:
\begin{itemize}
\item[\rm (a)] If $D^*$ is onto then $AH$ generates an exponentially stable semigroup. 
\item[\rm (b)] If $\begin{smallbmatrix}
G\\D
\end{smallbmatrix}$ is injective then $A$ generates a strongly stable semigroup. 
\end{itemize}
\end{proposition}
\begin{proof}
First, note that by Lemma~\ref{lem:stabwithH} we can assume that $H$ is the identity.  To show exponential stability, we first consider the case that $G=0$ and $R=r_0I_X$, i.e.
\begin{align*}
%\label{eq:A_simple}
A=\begin{bmatrix} 0&-D^*\\D&-r_0 I_X
\end{bmatrix}.
\end{align*}
In this case, we verify the characterization of exponential stability given by Proposition~\ref{prop:many_stable}~(d). By Lemma \ref{lem:2x2_inv} we have that $z\in\rho(A)$ if  $z(z+r_0)\notin(-\infty,0]$. This holds in particular if $\re z\geq 0$ and $z\neq 0$. Thus,
\[
\sigma(A)\subseteq\{z\in\C~|~\re z<0\}\cup\{0\}.
\]
Next, we show that $0\notin\sigma(A)$ holds. Since $D^*$ is onto $\ker D=(\ran D^*)^\perp=\{0\}$ implies that $D$ is injective. Hence, by Lemma~\ref{lem:spec}, $A$ is injective. Furthermore, using the closed range theorem we have that $\ran D$ is closed, and using the orthogonal decomposition $X=\ran(D)\oplus \ker D^*$ it follows that $\ran(D^*D)=\ran D^*=X$. This can be used to conclude the surjectivity of $A$. To this end let $(y_1,0)\in X\times X$ be given. Then there exists $Dx_1\in\dom D^*$ such that $y_1=-D^*Dx_1$. Hence for all $y_1\in X$ there exists $x_1\in X$ such that $A(r_0x_1,Dx_1)=(y_1,0)$ holds. This implies $\ran A\supseteq X\times \{0\}$. On the other hand, $X=\ran D\oplus \ker D^*$ implies that for $y_2=y_2^1+y_2^2$ with $y_2^1\in\ran D$ we have $A(y_2^1,-r_0^{-1}y_2^2)=(0,y_2)$. Hence $A$ is surjective which finally leads to $0\notin\sigma(A)$.

Furthermore, the resolvent along the imaginary axis for $\beta\neq 0$ is given by 
\begin{align}
   \nonumber  &~~~~\begin{bmatrix}
(\i\beta +r_0)I_X & D^*\\-D&\i\beta I_X
\end{bmatrix}^{-1}
\\&=\begin{bmatrix}
\i\beta &-D^*\\D&(\i\beta+r_0)I_X
\end{bmatrix}\begin{bmatrix}
((-\beta^2+\i\beta r_0)I_X+D^*D)^{-1}&0\\0&((-\beta^2+\i\beta r_0)I_X+DD^*)^{-1} \nonumber
\end{bmatrix}
\\&=\begin{bmatrix}
\i\beta ((-\beta^2+\i\beta r_0)I_X+D^*D)^{-1}&-D^*((-\beta^2+\i\beta r_0)I_X+DD^*)^{-1}\\D((-\beta^2+\i\beta r_0)I_X+D^*D)^{-1}&(\i\beta+r_0)((-\beta^2+\i\beta r_0)I_X+DD^*)^{-1} \label{eq:block_resolvent}
\end{bmatrix}.
\end{align}
We proceed by estimating the four block entries of the operator in \eqref{eq:block_resolvent}. First, the self-adjointness of $D^*D$ implies the resolvent estimate, see e.g.\ \cite[Chapter V,  \S  3.5]{Kat95}
\[
\|\i\beta ((\beta^2-\i\beta r_0)I_X-D^*D)^{-1}\|\leq \frac{|\i\beta|}{\text{dist}(\beta^2-\i\beta r_0,\sigma(D^*D))}\leq \frac{|\beta|}{|\beta|r_0}\leq \frac{1}{r_0},
\]
where $\rm{dist}(\beta^2-\i\beta r_0,\sigma(D^*D))$ denotes the distance between the point $\beta^2-\i\beta r_0\in\C$ and the closed set $\sigma(D^*D)\subseteq\C$. Similarly, by replacing $D$ with $D^*$, we can estimate
\[
\|(\i\beta+r_0) ((\beta^2-\i\beta r_0)I_X-DD^*)^{-1}\|\leq \frac{|\i\beta+r_0|}{\text{dist}(\beta^2-\i\beta r_0,\sigma(DD^*))}\leq \frac{|\i\beta+r_0|}{|\beta|r_0}\leq  \frac{1+|\beta|^{-1}r_0}{r_0},
\]
which is bounded for large values of $\beta$. It remains to bound the off-diagonal entries. To this end, consider $\gamma\in\C$ with $-\gamma\in\rho(D^*D)$ and we have for all $x\in X$
\begin{align*}
\|D(\gamma I_X+D^*D)^{-1}x\|^2&=\langle D(\gamma I_X+D^*D)^{-1}x,D(\gamma I_X+D^*D)^{-1}x\rangle\\&=\langle D^*D(\gamma I_X+D^*D)^{-1}x,(\gamma I_X+D^*D)^{-1}x\rangle\\&=\langle x,(\gamma I_X+D^*D)^{-1}x\rangle-\gamma \langle (\gamma I_X+D^*D)^{-1}x,(\gamma I_X+D^*D)^{-1}x\rangle.
\end{align*}
To estimate the off-diagonal entries, we choose $\gamma=-\beta^2+\i\beta r_0$. Then we already showed that $
\|(\gamma I_X+D^*D)^{-1}\|$ is bounded for sufficiently large $\beta$ and independently of $\beta$. Furthermore, we have according to the resolvent estimate for all $x\in X$
\begin{align*}
\left|\gamma \langle (\gamma I_X+D^*D)^{-1}x,(\gamma I_X+D^*D)^{-1}x\rangle\right|&\leq |\gamma|\|(\gamma I_X+D^*D)^{-1}\|^2\|x\|^2\\&\leq \frac{|-\beta^2+\i\beta r_0|}{|\beta|^2r_0^2}\|x\|^2\\
&\leq \frac{1+|\beta|^{-1}r_0}{r_0^2}\|x\|^2.
\end{align*}
In combination, this shows that the lower off-diagonal entry is uniformly bounded for large values of $\beta$. Similarly, one can show the uniform boundedness of the upper off-diagonal entry of \eqref{eq:block_resolvent} for large values of $\beta$. From this, we can conclude the boundedness of the resolvent uniformly in $\beta$. Hence, Proposition~\ref{prop:many_stable}~(d) provides us the exponential stability in this case. 

Therefore, by Proposition~\ref{prop:lyap} there exists a solution $P\in \eL(X)$ to the Lyapunov inequality \eqref{lyap_strict} satisfying
\[
\langle Ax,Px\rangle+\langle Px,Ax\rangle=\langle \begin{smallbmatrix}
0&-D^*\\D&-r_0I_X
\end{smallbmatrix},Px\rangle+\langle Px,\begin{smallbmatrix}
0&-D^*\\D&-r_0I_X
\end{smallbmatrix}x\rangle \leq -\langle x,x\rangle.
\]
This, together with assumption (A4) and $G\geq 0$ implies
\[
\langle \begin{smallbmatrix}
-G&-D^*\\D&-R
\end{smallbmatrix},Px\rangle+\langle Px,\begin{smallbmatrix}
-G&-D^*\\D&-R
\end{smallbmatrix}x\rangle \leq -\langle x,x\rangle 
\]
and therefore by Proposition~\ref{prop:lyap} the operator $A$ generates an exponentially stable semigroup.

We continue with the proof of (b). By Proposition~\ref{prop:lyap_RG}, $A$ generates a contraction semigroup. Applying Lemma~\ref{lem:stabwithH} shows that $A$ generates a contraction semigroup, which is in particular bounded. To prove the strong stability of the semigroup generated by $A$ we apply Proposition~\ref{prop:many_stable} (c). 
From the proof of part (a) it follows that the intersection $\sigma(A)\cap\i\R$ contains at most $0$ and is therefore countable. Furthermore, this implies for all  $\lambda\in\i\R\setminus\{0\}$ that the folllowing holds 
\[
(\ker(A^*-\lambda))^\perp=\overline{\ran(A-\overline{\lambda})}=X\times X.
\]
This implies $\ker(A^*-\lambda)=\{0\}$ and therefore $A^*$ has no eigenvalues $\lambda\in\i\R\setminus\{0\}$. To show the strong stability of the semigroup generated by $A$ it remains to show that $A^*$ is injective. If $\begin{smallbmatrix}
G\\D
\end{smallbmatrix}$ is injective then by Lemma~\ref{lem:spec}, $A^*$ is injective. %Furthermore, since $H$ is invertible it is well-known that we have $(AH)^*=H^*A^*$ and therefore $(AH)^*$ is injective, see e.g.\ \cite[Lemma 3.5.5]{Sta05}. 
Hence, by Proposition \ref{prop:many_stable} (c), $A$ generates a strongly stable semigroup.
\end{proof}

Another sufficient condition for exponential stability can be obtained under compact resolvent assumptions, see also \cite{Skr21} for related results. Here we will say that an eigenvalue $\lambda$ of an operator $A$ is \emph{isolated} if it is an isolated point of $\sigma(E, A)$. Furthermore, an eigenvalue $\lambda$ of $A$ has \emph{finite multiplicity} if $\dim\ker(A-\lambda I)<\infty$ holds. 

\begin{proposition}
\label{prop:compact}
Let $D$ be a closed densely defined operator in a Hilbert space $X$ such that $(I_X+DD^*)^{-1}$ and $(I_X+D^*D)^{-1}$ are compact and let $R,G\in \eL(X)$ be nonnegative and $H\in \eL(X\times X)$ uniformly positive. Then the spectrum of $A=\begin{smallbmatrix}
-G&-D^*\\D&-R
\end{smallbmatrix}$ consists of isolated eigenvalues with finite multiplicity. If $\begin{smallbmatrix}
G\\D
\end{smallbmatrix}$ is injective then $AH$ generates an exponentially stable semigroup.
\end{proposition}
\begin{proof}
It is well-known, see e.g.\ \cite[Theorem 3.49]{Kat95}, that $(I_X+D^*D)^{-1/2}$ is compact and that $D(I_X+D^*D)^{-1/2}$ is bounded see \cite[Lemma 5.8]{Sch12} and hence $D(I_X+D^*D)^{-1}$ is compact. Replacing $D$ by $D^*$ yields the compactness of $D^*(I_X+DD^*)^{-1}$. In summary, this implies the compactness of each of the block entries of 
\[
\begin{bmatrix}
-I_X &-D^*\\D&-I_X
\end{bmatrix}^{-1}=\begin{bmatrix}
-I_X &D^*\\-D&-I_X
\end{bmatrix}\begin{bmatrix}
(I_X+D^*D)^{-1}&0\\0&(I_X+DD^*)^{-1}
\end{bmatrix}
\]
and hence of the whole block operator. 
Hence $\left(\begin{smallbmatrix}
0&-D^*\\D&0
\end{smallbmatrix}-\lambda\right)^{-1}$ is compact for all $\lambda\in\rho\left(\begin{smallbmatrix}
0&-D^*\\D&0
\end{smallbmatrix}\right)$. This together with the resolvent identity 
\begin{align*}
\left(\begin{bmatrix}
0&-D^*\\D&0
\end{bmatrix}-\lambda\right)^{-1} &=\left(A+\begin{bmatrix}
G&0\\0&R
\end{bmatrix}-\lambda\right)^{-1} \\&=(A-\lambda)^{-1}-\left(\begin{bmatrix}
0&-D^*\\D&0
\end{bmatrix}-\lambda\right)^{-1}\begin{bmatrix}
G&0\\0&R
\end{bmatrix}(A-\lambda)^{-1}
\end{align*}
implies that $A$ has compact resolvent. By \cite[Theorem III.6.29]{Kat95} it follows that the spectrum of $A$ consists of isolated eigenvalues with finite multiplicity. Furthermore, $\begin{smallbmatrix}
G\\D
\end{smallbmatrix}$ injective implies with Lemma~\ref{lem:spec} that $A$ is injective and therefore $0\in\rho(A)$. Hence the exponential stability follows as in the proof of Proposition \ref{prop:exp_stable}.
\end{proof}

In the following, we consider the case where $D$ is not injective or $D^*$ is not surjective and restrict the operator $A$ to the Hilbert space $\ker A^\perp=\overline{\ran\begin{bmatrix}G&D^*\end{bmatrix}}\times X$. We define  
\[
\hat A:=A|_{\dom A\cap \ker A^\perp}
\] 
and show that the restriction $\hat A$ also generates a semigroup $(\hat T(t))_{t\geq 0}$. 

Furthermore, if $A$ fulfills (A4), then $\hat A$ fulfills (A4). Hence, if $D$ is not injective, this restriction allows us to conclude stability for the semigroup generated by $\hat A$. To this end, we will use the orthogonal projector $P_{\ker A}: X^2\rightarrow\ker A$ onto the kernel of $A$.

\begin{corollary}
\label{cor:hypocoercive}
Let $A=\begin{smallbmatrix}-G&-D^*\\D&-R\end{smallbmatrix}$ be given as above satisfying (A4)  and generating the semigroup $(T(t))_{t\geq 0}$. Then $\hat A:=A|_{\dom A\cap \ker A^\perp}$ generates a strongly stable semigroup $(\hat T(t))_{t\geq 0}$. Furthermore, if one of the following conditions holds
\begin{itemize}
    \item[\rm (i)] $D^*$ has closed range, or
    \item[\rm (ii)] $(I_X+D^*D)^{-1/2}$ and  $(I_X+DD^*)^{-1/2}$ are compact 
\end{itemize}
 then $(\hat T(t))_{t\geq 0}$ is exponentially stable and there exist $\alpha,M>0$ such that for all $x_0\in X\times X$, $t\geq 0$ the following holds 
\[
\|T(t)x_0-P_{\ker A}x_0\|\leq Me^{-\alpha t}\|(I-P_{\ker A})x_0\|.
\]
\end{corollary}
\begin{proof}
First it is shown that the semigroup $(T(t))_{t\geq 0}$ generated by $A$ fulfills
\begin{align}
\label{eq:gen_kerA}
T(t)=\begin{bmatrix}
\tilde T(t)&0\\0& I_{\ker A}
\end{bmatrix}\quad  \text{in $X\times X=\ker A^\perp\oplus \ker A$,}
\end{align}
where $\tilde T(t):=T(t)|_{\ker A^\perp}$ is again strongly continuous. If $x\in\ker A$, then the function $t\mapsto T(t)x$ is a classical solution $x(\cdot)$ with $x(0)=x$. Furthermore,  $\ker A$ is a closed $A$-invariant subspace which implies  $x(t)=T(t)x\in\ker A$ meaning that $\dot x(t)=Ax(t)=0$ and therefore $x(t)=x(0)=x$. Hence $T(t)x=x$ holds for all $x\in\ker A$. Furthermore, $A^*$ generates the strongly continuous semigroup $\{T(t)^*\}_{t\geq 0}$. If we apply the previous argument to $\{T(t)^*\}_{t\geq 0}$ with generator $A^*$ we obtain 
\[
T(t)^*y\in\ker A^*\quad \text{for all $y\in\ker A^*$.}
\]
By Lemma \ref{lem:spec} we have $\ker A^*=\ker A$ which implies for all $x\in(\ker A)^\perp=(\ker A^*)^\perp$ that
\[
\langle T(t)x,y\rangle=\langle x,T(t)^*y\rangle=0
\]
holds. Hence $T(t)x\in(\ker A)^\perp$ holds for all $x\in(\ker A)^\perp$. Finally, this proves \eqref{eq:gen_kerA}. Then by the generator definition, this implies that $A=\begin{smallbmatrix}
\hat A&0\\0& 0
\end{smallbmatrix}$ with respect to the decomposition $X\times X=\ker A^\perp\oplus \ker A$ and $\tilde T(t):=\hat T(t)$.

To show strong and exponential stability, we consider again the case $G=0$. In this case $\ker A=\ker D$ and therefore $\hat A$ is given in the Hilbert space $\hat X=\overline{\ran D^*}\times X$ by 
\[
\hat A=\begin{bmatrix} 0&\begin{bmatrix}-D^*|_{\overline{\ran D}} &0 \end{bmatrix}\\ \begin{bmatrix} D|_{\overline{\ran D^*}} \\ 0\end{bmatrix} &-R\end{bmatrix}.
\] 
 In particular, $D|_{\overline{\ran D^*}}$ is injective and therefore Proposition \ref{prop:exp_stable} (b) implies that $\hat A$ generates a strongly stable semigroup. If the condition (i) holds, then $D^*|_{\overline{\ran D}}$ is onto and therefore Proposition~\ref{prop:exp_stable} (a) implies that $\hat A$ generates an exponentially stable semigroup. Furthermore, if  (ii) is fulfilled then Proposition~\ref{prop:compact} yields the exponential stability of $\hat A$. As we argued before this then implies strong stability and exponential stability of $\hat A$ for arbitrary nonnegative $G\in\eL(X)$.

Finally, using the decomposition $x_0=P_{\ker A}x_0+P_{\ker A^\perp}x_0$ we find that 
\[
\|T(t)x_0-P_{\ker A}x_0\|=\|\hat T(t)P_{\ker A^\perp}x_0\|\leq Me^{-\alpha t}\|P_{\ker A^\perp}x_0\|
\]
for some $M,\alpha>0$ and all $t\geq 0$ since $\hat T$ is exponentially stable by assumption.
\end{proof}

In applications to transport problems, the operator $D$ is a derivative operator defined on a  Sobolev space and hence under suitable assumptions on the domain, it is a Fredholm operator. In this case, $\ran D$ is closed and by the closed range theorem $\ran D^*$ is closed as well. Additionally, the Sobolev embedding theorems imply the compactness of the resolvents $(I_X+D^*D)^{-1/2}$ and  $D(I_X+D^*D)^{-1/2}$. In this context, it was shown in \cite[Theorem 3.5]{EggK18} that if (A4) holds then the solutions of \eqref{def:bctrl} with $u=0$ are exponentially stable using Galerkin approximation, see also \cite{EggKLSMM18} for related results for non-linear replacements for $R$ and $G$.

In the following we study the solvability of the boundary control problem for non-zero input functions $u$. To guarantee the existence of solutions to the boundary control problem \eqref{def:bctrl}, we assume that the following holds:
\begin{itemize}
\item[(A5)]
The triplet $\{\Xc,\Gamma_0,\Gamma_1\}$ is a boundary triplet for the symmetric operator $S:=\begin{smallbmatrix}0& \i D\\ \i D& 0\end{smallbmatrix}$, which means that $\Gamma_0,\Gamma_1:\dom S^*\rightarrow\Xc$ are assumed to be linear mappings that fulfill the following two conditions
\begin{itemize}
    \item[(i)]  $[\Gamma_0,\Gamma_1]:\dom S^*\rightarrow \Xc\times\Xc$ is surjective;
    \item[(ii)] the following abstract Green identity holds
    \[
\langle S^*f,g\rangle-\langle f ,S^*g\rangle=\langle\Gamma_1f,\Gamma_0g\rangle-\langle\Gamma_0f,\Gamma_1g\rangle,\quad f,g\in \dom S^*.
    \] 
\end{itemize}
\end{itemize}
The framework of boundary triplets was developed in \cite{Gorb89,Derk91} to describe the extensions of symmetric operators, see also \cite{Behr20}. 

\begin{lemma}
\label{lem:with_H}
Let $S$ be a closed densely defined symmetric operator in a Hilbert space $X$ and consider a uniformly positive operator $H\in\mathcal{L}(X)$. Then $SH$ is symmetric in $X$ with respect to the scalar product given by $\langle x,y\rangle_{H}:=\langle H x,y\rangle$ for all $x,y\in X$. Moreover, if  $\{\Xc,\Gamma_0,\Gamma_1\}$ is a boundary triplet for $S$ then  $\{\Xc,\Gamma_0H,\Gamma_1H\}$ is a boundary triplet for $SH$.
\end{lemma}

In the following, we obtain a power-balance equation for the classical solutions of \eqref{def:bctrl}. 
\begin{proposition}
\label{prop:pbe}
Let $\mathfrak{A}$ and $H$ be given by (A1)-(A3) and let $\{\Xc,\Gamma_0,\Gamma_1\}$ be a boundary triplet as in (A5). Then for all $u\in C^2([0,\infty),\mathcal{X})$ and $z_0\in X\times X$ with $\Gamma_0H z_0=u(0)$  exists $z\in C^1([0,\infty),X\times X)\cap C([0,\infty),\dom (\mathfrak{A}H))$ satisfying
\begin{align}
\label{eq:A_bctrl}
\dot z(t)=\mathfrak{A}H z(t), \quad u(t)=\Gamma_0H z(t),\,\, t\geq 0,\quad  z(0)=z_0.
\end{align}
Furthermore, if $y(t):=-\i\Gamma_1H z(t)$ is continuous the following power balance equations hold 
\begin{align}
\label{eq:PBE}
\re \langle\mathfrak{A}H z(t),H z(t)\rangle&=\re \langle u(t),y(t)\rangle-\left\langle\begin{bmatrix}G&0\\0&R
\end{bmatrix}H z(t),H z(t)\right\rangle,\\
\langle H z(t),z(t)\rangle-\langle H z(0),z(0)\rangle&\leq 2\int_0^t\re \langle u(\tau),y(\tau)\rangle d\tau-\left\langle\begin{bmatrix}G&0\\0&R
\end{bmatrix} H z(t), H z(t)\right\rangle.\label{scnd_PBE}
\end{align}
\end{proposition}
\begin{proof}
First, we consider the operator
$\mathfrak{J}:=\begin{smallbmatrix} 0& D^* \\   D^*&0\end{smallbmatrix}H$  
and the boundary control system 
\begin{align*}
%\label{eq:J_bctrl}
\tfrac{d}{d t}x(t)=\i\mathfrak{J}H x(t),\quad u(t)=\Gamma_1H x(t),\quad y(t)=\Gamma_0H x(t).
\end{align*}
By Lemma~\ref{lem:with_H} the triplet $\{\Xc,\Gamma_0H,\Gamma_1H\}$ is a boundary triplet for $\i\mathfrak{J}H$. From \cite[Theorem 5.2]{MalS07} and since $\mathfrak{A}H$ is a bounded perturbation of $\mathfrak{J}H$ and \cite[Lemma 2.6]{MalS06} we have that for $u\in C^2([0,\infty),\Xc)$ and $z_0\in X$ with $\Gamma_1H z_0=u(0)$ there is a unique classical solution $z\in C^1([0,\infty),X)\cap C([0,\infty),\dom( \mathfrak{A}H))$ 
satisfying \eqref{eq:A_bctrl} and $y(t):=\Gamma_1H z(t)$ is continuous.  Invoking the abstract Green identity
\[
\langle i\mathfrak{J}x_1,x_2\rangle-\langle x_1,i\mathfrak{J}x_2\rangle=\langle \Gamma_1x_1,\Gamma_0x_2\rangle-\langle \Gamma_0x,\Gamma_1y\rangle,\quad x_1,x_2\in\dom\mathfrak{J}.
\]

The following dissipation inequality holds for all $t\geq 0$
\begin{align*}
%\label{eq:PBE}
\re \langle\mathfrak{A}H z(t),H z(t)\rangle &=\re \langle\Gamma_1H z(t),\Gamma_0H z(t)\rangle- \left\langle\begin{bmatrix}G&0\\0&R
\end{bmatrix}H z(t),H z(t)\right\rangle\\&=\re \langle u(t),y(t)\rangle- \left\langle\begin{bmatrix}G&0\\0&R
\end{bmatrix}H z(t),H z(t)\right\rangle.
\end{align*}
The power balance equation \eqref{scnd_PBE} follows after replacing $\tfrac{d}{d t}z(t)=\mathfrak{A}H z(t)$ in \eqref{eq:PBE} and integrating this equation.
\end{proof}

\section{Structure preserving interconnection of pH subsystems}
\label{sec:interconnect}
In this section, we describe generalized Kirchhoff interconnection conditions which will preserve the considered system class. To this end, we show that these interconnections lead to skew-adjoint restrictions of the maximal operator $A$ as in (A1) leading to an operator satisfying (A3). However, here we will rather consider self-adjoint extensions of symmetric operators, to stay in the setting of boundary triplets.  By multiplying with the imaginary unit we transform symmetric in skew-symmetric operators and vice versa.

Let $S$ be a symmetric operator in a Hilbert space $X$ with boundary triplet $\{\Xc,\Gamma_0,\Gamma_1\}$ for $S^*$. Furthermore, let $\Theta \subseteq\Xc\times\Xc$ be a subspace, also called \emph{linear relation}, then extensions of $S$ can be parametrized as follows
\[
\dom S_{\Theta}:=\{x\in\dom S^* ~|~ (\Gamma_0x,\Gamma_1x)\in\Theta\}.
\]
One can show that properties of the relation $\Theta$ translate to the operator $S_{\Theta}$. To this end, recall that the \emph{adjoint} of a linear relation $\Theta$ in $\Xc\times\Xc$ is given by 
\[
\Theta^*:=\{(x,y)\in\Xc\times\Xc ~|~ \langle v,x\rangle=\langle u,y\rangle~\text{for all $(u,v)\in\Theta$}\}.
\]
Then one has 
\begin{align*}
%\label{eq:adjoint}
S_{\Theta}^*=S_{\Theta^*},
\end{align*}
see e.g.\ \cite[Theorem 2.1.3]{Behr20}, which implies that $\Theta$ is self-adjoint if and only if $S_{\Theta}$ is self-adjoint. 
Furthermore, Proposition \ref{prop:many_stable} (b) and Remark \ref{rem:stable} (c) together with the dissipativity of $\Theta$ and $\Theta^*$ imply that $S_{\Theta}$ generates a contraction semigroup.

We will use the above-described way to obtain self-adjoint extensions for our class of pH operators on the network.  Here we consider now several pH subsystems which are given by the following closed densely defined symmetric operators $S_j$, $j=1,\ldots,m$ in the Hilbert spaces $X_j\times X_j$ given by
\begin{align}
\label{def_Sj}
S_{j}:=\begin{bmatrix}
0&\i D_{j}\\\i D_{j}&0
\end{bmatrix},
\end{align}
and assume for the coupling the following stronger assumption than (A5).
\begin{itemize}
\item[(A5')] There exists a boundary triplet $\{\Xc,\Gamma_0^j,\Gamma_1^j\}$ for $S_j^*$ of the following form $\Xc=\hat \Xc\times\hat\Xc$ for some Hilbert space $\hat \Xc$ and the abstract boundary mappings are given by 
\begin{align*}
\Gamma_0^j:\dom S_j^*\rightarrow\Xc,\quad x_j=(x_0^j,x_1^j)\mapsto\begin{bmatrix}
\Gamma_{0,0}^jx_0^j\\ \Gamma_{0,1}^jx_0^j
\end{bmatrix},\\
\Gamma_1^j:\dom S_j^*\rightarrow\Xc,\quad x_j=(x_0^j,x_1^j)\mapsto\begin{bmatrix} \Gamma_{1,0}^jx_1^j\\ - \Gamma_{1,1}^jx_1^j\end{bmatrix}
\end{align*}
using the component mappings $\Gamma_{k,l}^j:\dom D_j^*\rightarrow\hat \Xc,\quad k,l=0,1$.
\end{itemize}

We show in Section \ref{sec:TL} that for electrical transmission lines, possible choices of $\Gamma_{0,0}^jx_0$ and $\Gamma_{0,1}^jx_0$ are the voltages at the left and right endpoint of the line, respectively, and by choosing $\Gamma_{1,0}^jx_1$ and $\Gamma_{1,1}^jx_1$ as the currents at the endpoints of the line. 

Next, we introduce the direct sum operator 
\begin{align*}
%\label{def:S_opl}
S_{\oplus}:=\begin{bmatrix}0&\i \bigoplus_{j=1}^mD_{j}\\\i \bigoplus_{j=1}^mD_{j}&0\end{bmatrix}
\end{align*}
which is closed and densely defined and symmetric in the direct sum space $X_{\oplus}\times X_\oplus$ where $X_\oplus:=\oplus_{j=1}^m X_j$. Then the direct sum triplet 
\[
\Xc_{\oplus}:=\bigoplus_{j=1}^m\Xc,\quad \Gamma_0^{\oplus}(x_j)_{j=1}^m:=(\Gamma_0^jx_j)_{j=1}^m,\quad \Gamma_1^{\oplus}(x_j)_{j=1}^m:=(\Gamma_1^jx_j)_{j=1}^m,\quad 
\]
is a boundary triplet for $S_{\oplus}^*$ because the surjectivity and abstract Green identity  carries over from these properties of the summand triplets.

Next, we couple the edge operators based on a given graph $\Gc=(V,E)$ where $V$ is a finite set of vertices and $E$ is a finite set of edges given by the intervals $e_j=[0,1]$. The idea is that for the $j$th edge the evaluations $\Gamma_{0,0}^jx_0^j$ and $\Gamma_{0,1}^jx_0^j$ will be identified with the endpoints which correspond to two distinct vertices $v,w\in V$. This will be denoted by $(j,0)\sim v$ or $(j,1)\sim w$. Furthermore,  the notion $e_j\sim v$ means that $e_j$ has $v$ as one of its vertices. Another notion that will be used is $\sgn(e_j,v)\in\{-1,1\}$ for an edge $e_j\in E$ and a vertex $v$ which is $\sgn(e,v)=1$ if $(j,0)\sim v$ and $-1$ if $(j,1)\sim v$.

For a given graph $\Gc=(V,E)$, we want to consider vertex based generalized Kirchhoff coupling conditions and show that they lead to self-adjoint extensions of $S_{\oplus}$. Here we consider at every node $v\in V$
\begin{align}
\label{eq:cont}
\Gamma_{0,j_2}^{j_1}x_0^{j_1}=\Gamma_{0,k_2}^{k_1} x_0^{k_1}\quad  \text{for all  $j_1,k_1=1,\ldots,m$,~ $j_2,k_2\in\{0,1\}$ with $(j_1,j_2)\sim v$ and $(k_1,k_2)\sim v$}
\end{align}
and for the common value at the vertices, we will write for all $v\in V$
\[
(\Gamma_0^{\oplus}x)(v):=\begin{cases}\Gamma_{0,0}^jx_0^j & \text{if $(j,0)\sim v$},\\  \Gamma_{0,1}^jx_0^j & \text{if $(j,1)\sim v$.} \end{cases} 
\]

In the case of transmission lines, this can be interpreted as equaling the line voltages at the interconnection points and resembles the method of modified nodal analysis for lumped parameter models of electrical circuits \cite{HoRB75}. This leads to the following node-based extension of $S_{\oplus}$
\[
\dom S_{\oplus,cnt}=\{x\in \dom S_{\oplus}^*~|~ \text{$x$ fulfills \eqref{eq:cont}}\}
\]
and 
\[
\dom D_{node}=\{x_0\in \dom \oplus_{j=1}^m(D_{j})^*~|~ \text{$x_0$ fulfills \eqref{eq:cont}}\}.
\]

Below we formulate the first main result of this section, a similar result was obtained in \cite[Theorem 4.2]{GernTrun21} which was stated only for finite dimensional $\Xc$ but can easily be extended to the  general case.
\begin{proposition}
\label{prop:node}
Let $S_1,\ldots,S_m$ be closed densely-defined symmetric operators given by \eqref{def_Sj}. Then for the direct sum operator $S_{\oplus}$  the following holds:
\begin{itemize}
\item[\rm (a)]
The triplet given by $\left\{\hat \Xc^{|V|},\Gamma_0^V, \Gamma_1^V\right\}$ with $\Gamma_0^V, \Gamma_1^V:\dom S_{cnt,\oplus}\rightarrow \hat \Xc^{|V|}$ and
\begin{align*}
%\label{eq:graph_bt}
\Gamma_0^Vx:=\{(\Gamma_0^{\oplus}x)(v_j)\}_{j=1}^{|V|},\quad \Gamma_1^Vx:=\big\{\sum_{e_l\sim v_j, (l,k)\sim v_j}{\rm sgn}(e_l,v_j) \Gamma_{1,k}^l x_1^l\big\}_{j=1}^n, 
\end{align*}
is a boundary triplet for the operator $S_{cnt,\oplus}$.
\item[\rm (b)] The operator given by
\[
\dom S_{node}:=\{x\in \dom S_{\oplus}^*~|~\text{$x$ fulfills \eqref{eq:cont} and $\Gamma_1^Vx=0$}\},\quad S_{node}x:=S_{\oplus}^*x,
\]
is a self-adjoint extension of $S_{\oplus}$  and given by
\[
S_{node}=\begin{bmatrix}0&iD_{node}^*\\iD_{node}&0
\end{bmatrix}.
\]
The operator given by
\[
\dom S_{0}:=\{x\in \dom S_{\oplus}^*~|~\text{$x$ fulfills \eqref{eq:cont} and $\Gamma_0^Vx=0$}\},\quad S_0x:=S_{\oplus}^*x,
\]
is a self-adjoint extension of $S_{\oplus}$  and given by
\[
S_{0}=\begin{bmatrix}0&\i D_{0}^*\\\i D_{0}&0
\end{bmatrix},
\]
where $D_0$ is an extension of $\oplus_{j=1}^mD_j$ with domain 
\[
\dom D_0:=\{ x_0\in \dom\oplus_{j=1}^mD_j~|~\text{$\Gamma_{0,j}^{k}x_0^k=0$ for all $j=0,1, k=1,\ldots,m$}  \}.
\]

\item[\rm (c)] Let $L=\diag(L_v)\in \eL (\hat \Xc^{|V|})$, $L_v\in \eL(\hat \Xc)$, then the operator given by 
\[
\dom S_{\oplus,L}:=\{x\in \dom S_{\oplus}^*~|~\text{$x$ fulfills \eqref{eq:cont} and $\Gamma_1^Vx=L\Gamma_0^Vx$}\}, \quad S_{\oplus,L}x:=S_{\oplus}^*x
\]
is an extensions of $S_{\oplus}$ which satisfies
\[
\im (S_{\oplus,L}x,x)=\im (L\Gamma_0^Vx,\Gamma_0^Vx),\quad \text{for all $x\in\dom S_{\oplus,L}$}.
\]
If $-\i L$ and $\i L^*$ are dissipative then $S_{\oplus,L}$ generates a contraction semigroup.
\end{itemize}
\end{proposition}
\begin{proof}
We have to show that $(\Gamma_0^V,\Gamma_1^V):\dom S_{\oplus,cnt}\rightarrow \hat \Xc^{|V|}\times \hat \Xc^{|V|}$ is surjective and that the abstract Green identity holds. The surjectivity follows from the surjectivity of the mappings $(\Gamma_0^l,\Gamma_1^l):\dom S_l^*\rightarrow\Xc\times\Xc$ given by \eqref{def_Sj}. The abstract Green identity has the following form
\begin{align*}
&~~~~\langle S^*(x_j)_{j=1}^m,(y_j)_{j=1}^m\rangle-\langle(x_j)_{j=1}^m,S^*(y_j)_{j=1}^m\rangle\\&=\sum_{j=1}^m\langle S_j^*x_j,y_j\rangle-\langle x_j,S_j^*y_j\rangle
\\
&=\sum_{j=1}^m\langle\Gamma_1^jx_j,\Gamma_0^jy_j\rangle-\langle\Gamma_0^jx_j,\Gamma_1^jy_j\rangle\\
&=\sum_{j=1}^m\left\langle\begin{bmatrix}\Gamma_{1,0}^jx_1^j\\- \Gamma_{1,1}^jx_1^j\end{bmatrix},\begin{bmatrix}\Gamma_{0,0}^jy_0^j\\\Gamma_{0,1}^jy_0^j\end{bmatrix}\right\rangle-\left\langle\begin{bmatrix}\Gamma_{0,0}^jx_0^j\\\Gamma_{0,1}^jx_0^j\end{bmatrix},\begin{bmatrix}\Gamma_{1,0}^jy_1^j\\-\Gamma_{1,1}^jy_2^j\end{bmatrix}\right\rangle\\
&=\sum_{v\in V}\langle\Gamma_1^Vx,\Gamma_0^Vy\rangle-\langle\Gamma_0^Vx,\Gamma_1^Vy\rangle,
\end{align*}
where we used in the last step that $x,y\in\dom S_{\oplus, cnt}$.

To prove (b), consider $x=(x_0,x_1)\in\dom S_{\oplus}^*$. Then $\Gamma_0^V$ acts only on $x_0$ whereas $\Gamma_1^V$ acts only on $x_1$. We show that the self-adjoint relation $\Theta\subseteq\Xc_{\oplus}^2$ which satisfies  $S_{node}=S_{\Theta}$ is given by 
\[
\Theta=\Xc_{cnt}\times \Xc_{cnt}^\perp, \quad \Xc_{cnt}:=\oplus_{v\in V}\Xc_v
%\Theta=(\Gc_{cnt}\times \{0\})\oplus (\{0\}\times \Gc_{cnt}^\perp).
\] 
where $\Xc_v$ is a closed subspace of $\Xc_{\oplus}=\hat\Xc^{2m}$ which will be defined in the following. To this end, we identify $\Xc_{\oplus}$ with the direct sum $\oplus_{(j,l)\in I}\hat \Xc$ over $I:=\{1,\ldots,m\}\times \{0,1\}$. Consequently, the elements of $\Xc_{\oplus}$ are be denoted by $\left(\begin{smallbmatrix}\hat x_0^j\\ \hat x_1^j\end{smallbmatrix}\right)_{j=1}^m$ and then $\Xc_v$ is given by 
\[
\Xc_v:=\left\{ \left(\begin{smallbmatrix}\hat x_0^j\\ \hat x_1^j\end{smallbmatrix}\right)_{j=1}^m\in\hat\Xc^{2m}~|~ \hat x^j_k=\hat y ~\text{for all $(j,k)\sim v$ and some $\hat y\in \hat\Xc$ and $\hat x^j_k=0$ otherwise}\right\}.
\]
Denoting by $J$ the flip-flop operator on $\Xc^{\oplus}\times \Xc^{\oplus}$ which maps $(x,y)\mapsto(-y,x)$, see \cite[Chapter 2]{Behr20}, we obtain
\[
\Theta^*=J\Theta^\perp=J(\Xc_{cnt}^{\perp}\times \Xc_{cnt})=\Theta.
\]
Because of the product structure of $\Theta$ the condition
\[
(\Gamma_0^\oplus x,\Gamma_1^\oplus x)=\left(\left(\begin{bmatrix}\Gamma_{0,0}^jx_0^j\\ \Gamma_{0,1}^jx_0^j\end{bmatrix}\right)_{j=1}^m,\left(\begin{bmatrix}\Gamma_{1,0}^jx_1^j\\ - \Gamma_{1,1}^jx_1^j\end{bmatrix}\right)_{j=1}^m\right)\in\Theta
\]
is equivalent to
\[
\Gamma_0^{\oplus}x\in\Xc_{cnt},\quad \Gamma_1^{\oplus}x\in\Xc_{cnt}^\perp.
\]
Hence $\Gamma_0^{\oplus}x\in\Xc_{cnt}$ is equivalent to $x$ satisfying \eqref{eq:cont}. If we write now 
\[
\Xc_{cnt}=\ran M,
\]
where $M$ is some bounded operator with closed range $\ran M$ which can be obtained from the decomposition  $\Xc_{cnt}=\oplus_{v\in V}\Xc_v$ and the definition of $\Xc_v$. Then we have
\[
\Xc_{cnt}^\perp=\ker M^*.
\]
Hence $\Gamma_1^{\oplus}x\in\Xc_{cnt}^\perp$ is equivalent to $\Gamma_1^Vx=0$. Furthermore, since the lower off-diagonal block is equal to $iD_{\oplus,cnt}$ and $S_{\oplus,node}$ is self-adjoint, the upper off diagonal operator must be equal to $(iD_{\oplus,cnt})^*$. To prove (c) we consider the relation 
\[
\Theta_L:=\{(x,\hat Lx+y)~|~ x\in\Xc_{cnt},~  y\in\Xc_{cnt}^\perp\},\quad (\hat Lx)_{(j,l)}:=\tfrac{1}{\deg v}L_vx_{(j,l)} \quad \text{if $v\sim(j,l)$},
\]
then we have $S_{\Theta_L}=S_{\oplus,L}$.
\end{proof}

In the context of electrical circuits there is not only a vertex based definition of the Kirchhoff laws but also a loop based definition, see e.g.\ \cite{GerHRS21,NedPS22,Rei14}. This construction will be used to define another class of extensions of $S_{\oplus}$. To this end, we assume that the underlying graph $\Gc=(V,E)$ is connected and consider a \emph{spanning tree}, i.e.\ a connected subgraph without cycles that contains all vertices of $\Gc$. Then for each edge $e\in E$ which is not contained in the spanning tree there is a unique cycle containing $e$ and having its remaining edges in the spanning tree. This cycle is called a \emph{fundamental cycle}. We add an orientation to each fundamental cycle which is taken as the orientation of the edge that is not part of the spanning tree. Using this orientation we define the following sign of edges $e\in E$ for fundamental cycles $\mathcal{C}$
\[
\sgn(e,\mathcal{C}):=\begin{cases}1&\text{if $e$ has the same orientation as $\mathcal{C}$,}\\-1&\text{if $e$ has a different orientation as $\mathcal{C}$,}\\ 0 & \text{if $e$ is not contained in $\mathcal{C}$.}\end{cases}
\]

For electrical circuits, the Kirchhoff voltage law implies that the sum of all cycle voltages with appropriate signs in each fundamental cycle is zero. Furthermore, similar to the vertex based formulation, we have a corresponding condition for the adjoint operator which means here that the currents at the nodes of every fundamental cycle $\mathcal{C}$ are constant and this constant value will be denoted by 
\begin{align}
\label{eq:const_loop}
\Gamma_1^{\oplus}x(\mathcal{C})=\Gamma_{1,k}^lx_2,\quad \text{for all $(k,l)\sim e\in\mathcal{C}$}.
\end{align}
These conditions define an operator
\[
\dom D_{loop}=\{x_2\in \dom \oplus_{j=1}^m(D_{j})^*~|~ \text{$x_2$ fulfills \eqref{eq:const_loop} for all fundamental cycles $\mathcal{C}$}\}
\]
and the following extension of $S_{\oplus}$
\[
\dom S_{cst,\oplus}:=\{x\in \dom S_{\oplus}^*~|~ \text{$x$ fulfills \eqref{eq:const_loop}}\}.
\]

\begin{proposition}
%\label{prop:loop}
Let $S_1,\ldots, S_m$ be closed densely-defined symmetric operators given by \eqref{def_Sj} and consider a graph $\Gc=(V,E)$ and set $C:=|E|-|V|+1$. Then the following holds:
\begin{itemize}
\item[\rm (a)]
The triplet given by $\left\{\hat \Xc^{C},\Gamma_0^E, \Gamma_1^E\right\}$ with
\begin{align*}
%\label{eq:graph_bt_edge}
\Gamma_0^Ex:=\{\Gamma_1^{\oplus}x_2(\mathcal{C}_j)\}_{j=1}^{C},\quad \Gamma_1^Ex:=\big\{\sum_{e_l\in\mathcal{C}_j, (k,l)\sim e_l}{\rm sgn}(e_l,\mathcal{C}_j)\Gamma_{k}^l x_1^l\big\}_{j=1}^C
\end{align*}
is a boundary triplet for the operator $S_{cst,\oplus}$. 
\item[\rm (b)] The operator given by 
\[
\dom S_{loop}:=\{x\in \dom S_{\oplus}^*~|~\text{$x$ fulfills \eqref{eq:const_loop} and $\Gamma_1^Ex=0$}\}
\]
is a self-adjoint extension of $S_{\oplus}$ and
\[
S_{loop}=\begin{bmatrix}0&\i D_{loop}\\\i D_{loop}^*&0
\end{bmatrix}.
\]
\item[\rm (c)] Let $L\in L(\Xc^{C})$ then the operator given by 
\[
\dom S_{\oplus,L}:=\{x\in \dom S_{\oplus}^*~|~\text{$x$ fulfills \eqref{eq:const_loop} and $\Gamma_1^Ex=L\Gamma_0^Ex$}\}
\]
is an extensions of $S_{\oplus}$ which satisfies
\[
\im (S_{\oplus,L}x,x)=\im (L\Gamma_0^Ex,\Gamma_0^Ex),\quad \text{for all $x\in\dom S_{\oplus,L}$}.
\]
\end{itemize}
\end{proposition}

\section{Application to distributed transmission line networks}
\label{sec:TL}
In this section, we apply our results to study the stability and passivity of electrical networks having distributed transmission lines. In the literature, the transmission lines are often modeled in form of $\pi$-equivalent models which provide accurate lumped models. However, distributed line models are necessary to accurately model higher frequency behavior, longer transmission lines or transmission lines with impurities.

To study a line of length $\ell>0$ we consider a spatial interval $[0,\ell]\subset\R$. Then the voltage and current in a single line are given as solutions to the telegraph equations 
\begin{align}
C(\xi)\tfrac{\partial U}{\partial t}(t,\xi)&=-\tfrac{\partial I}{\partial \xi}(t,\xi)-G(\xi)U(t,\xi),\nonumber\\
L(\xi)\tfrac{\partial I}{\partial t}(t,\xi)&=-\tfrac{\partial U}{\partial \xi}(t,\xi)-R(\xi)I(t,\xi), \label{eq:tele}\\
&U(0,\xi)=U^0(\xi),\quad I(0,\xi)=I^0(\xi),\quad t\geq 0\nonumber, 
\end{align}
where $I:[0,\infty)\times [0,\ell]\rightarrow\R^d$ is the current and $U:[0,\infty)\times [0,\ell]\rightarrow\R^d$ is the voltage across the transmission line. We allow values in $\R^d$ for $d>1$ to include three-phase models of transmission lines as well. Furthermore, the underlying Hilbert space is $X=L^2([0,\ell],\R^d)$ and we assume that the matrix-valued functions $C,G,L,R:[0,\ell]\rightarrow \C^{d\times d}$ standing for the capacitance, conductance, inductance and resistance have values in the positive semi-definite matrices and fulfill $\langle C(\cdot)x,x\rangle\geq c\|x\|^2$ and $\langle L(\cdot)x,x\rangle\geq l\|x\|^2$ with $c,l>0$, $\langle R(\cdot)x,x\rangle\geq 0$ and $ \langle G(\cdot)x,x\rangle \geq 0$ for all $x\in X$ and $R(\cdot),G(\cdot)\in L^{\infty}([0,\ell],\C^{d\times d})$.

\subsection*{Stability and passivity}
The telegraph equations \eqref{eq:tele} can be written as an abstract Cauchy problem in the following way
\[
\tfrac{d}{dt}z(t)=\mathfrak{A}Hz(t),\quad Hz(0)=(U^0,I^0)
\]
with $X\times X=L^2([0,\ell],\R^d)\times L^2([0,\ell],\R^d)$ and 
\[
\mathfrak{A}=\begin{bmatrix}
-G&D^*\\D^*&-R
\end{bmatrix},\quad \mathcal{H}(\xi)=\begin{bmatrix}C(\xi)^{-1}&0\\0&L(\xi)^{-1}\end{bmatrix},\quad z(t,\xi)=\begin{bmatrix}Q(t,\xi)\\ \phi(t,\xi)\end{bmatrix},\quad \mathcal{H}(\xi)z(t,\xi)=\begin{bmatrix}U(t,\xi)\\ I(t,\xi)\end{bmatrix},
\]
where $Q$ is the charge and $\phi$ is the flux and the skew-symmetric operator $D$ is given by
\begin{align}
\label{def_D_JZ}
Dx=\tfrac{\partial}{\partial\xi}x,\quad \dom D:=\{x\in L^2([0,\ell],\C^d) ~|~  x\in H^1([0,\ell],\C^d),~ x(0)=x(\ell)=0\}.
\end{align}

Below we show that for the transmission line, there exists a boundary triplet of the desired form. 
\begin{lemma}
\label{lem:bt_TL}
Let $D$ be given by \eqref{def_D_JZ} and $X_1=X_2=L^2([0,\ell],\C^d)$, and let $H=\begin{smallbmatrix}
H_1&0\\0&H_2
\end{smallbmatrix}$ with $H_i\in \eL(X_i)$, $i=1,2$ be uniformly positive. Then $S=\begin{smallbmatrix}
0&\i D\\\i D&0
\end{smallbmatrix}H$ is symmetric  in $X$ with respect to the scalar product given by $\langle H \cdot,\cdot\rangle$ and  $S^*$ has the boundary triplet 
\[
\left\{\C^{2d},(x_1,x_2)\mapsto\begin{bmatrix}
H_1x_1(0)\\ H_1x_1(\ell)
\end{bmatrix},(x_1,x_2)\mapsto\begin{bmatrix}
-\i H_2x_2(0)\\ \i H_2x_2(\ell)
\end{bmatrix}\right\}.
\]
\end{lemma}
\begin{proof}
We compute with the weighted scalar product $\langle H\cdot, \cdot\rangle$ 
\begin{align*}
\left\langle \begin{bmatrix}
0&\i D\\\i D&0
\end{bmatrix} \begin{bmatrix}H_1x_1\\ H_2x_2\end{bmatrix}, \begin{bmatrix}H_1y_1\\ H_2y_2\end{bmatrix}\right\rangle&=\langle \i DH_2x_2,H_1y_1\rangle+\langle \i DH_1x_1,H_2y_2\rangle\\
&=\langle \i H_2x_2(\ell),H_1y_1(\ell)\rangle-\langle \i H_2x_2(0),H_1y_1(0)\rangle\\
&-\langle \i H_2x_2,DH_1y_1\rangle -\langle \i H_1x_1,DH_2y_2\rangle \\
&+\langle \i H_1x_1(\ell),H_2y_2(\ell)\rangle-\langle \i H_1x_1(0),H_2y_2(0)\rangle\
\end{align*}
and 
\[
\left\langle \begin{bmatrix}H_1x_1\\ H_2x_2\end{bmatrix}, \begin{bmatrix}
0&\i D\\\i D&0
\end{bmatrix} \begin{bmatrix}H_1y_1\\ H_2y_2\end{bmatrix}\right\rangle=\langle H_1x_1,\i DH_2y_2\rangle+\langle H_2x_2,\i DH_1y_1\rangle.
\]

On the other hand it holds
\begin{align*}
&~~~\left\langle \begin{bmatrix} -\i H_2x_2(0)\\ \i H_2x_2(\ell)\end{bmatrix}, \begin{bmatrix} H_1y_1(0)\\ H_1y_1(\ell)\end{bmatrix}\right\rangle-\left\langle \begin{bmatrix} H_1x_1(0)\\ H_1x_1(\ell)\end{bmatrix}, \begin{bmatrix} -\i H_2y_2(0)\\ \i H_2y_2(\ell)\end{bmatrix} \right\rangle\\
&=-\langle \i H_2x_2(0) H_1y_1(0)\rangle+\langle H_1x_1(0),\i H_2y_2(0)\rangle+\langle \i H_2x_2(\ell),x_1y_1(\ell)\rangle-\langle H_1x_1(\ell),\i H_2y_2(\ell)\rangle.
\end{align*}
which shows the Green identity. The surjectivity is clear since each of the mappings is surjective and the ranges use independent variables. 
\end{proof}

There are now several choices for boundary triplets, each resulting in a different boundary control system. Two possible choices are 
\begin{align*}
\Gamma_0(x):=\begin{bmatrix}
U(0)\\ U(\ell)
\end{bmatrix},\quad \Gamma_1(x):=\begin{bmatrix}
-\i I(0)\\ \i I(\ell)
\end{bmatrix},\\
\hat \Gamma_0(x):=\begin{bmatrix}
 I(0)\\ I(\ell)
\end{bmatrix},\quad \hat \Gamma_1(x):=\begin{bmatrix}
-\i U(0)\\\i U(\ell)
\end{bmatrix}.
\end{align*}

This implies that the telegraph equations \eqref{eq:tele} with either current or voltage boundary inputs is a boundary control system of the form 
\[
\dot z(t)=\mathfrak{A}Hz(t),\quad u(t)=\Gamma_0z(t),\quad y(t)=\Gamma_1 y(t)
\]
and a power-balance equation similar to \eqref{eq:PBE} holds for control inputs $u\in C^2([0,T],\C^2)$. 

Proposition \ref{prop:many_stable} implies that the underlying semigroups are contraction semigroups. However, they are not strongly stable in general. The block operator $A$ associated with \eqref{eq:tele} for $G=0$ and the boundary conditions $I(0)=I(\ell)=0$ is not injective. If $R=0$ then the boundary conditions $U(0)=U(\ell)=0$ would imply a non-trivial kernel for $A$.

Consider now a network of transmission lines described by a graph $\Gc=(V,E)$. Then for each line $e\in E$ with spatial domain $[0,\ell_e]$ for some $\ell_e>0$, we consider in $X_e\times X_e=L^2([0,\ell_e],\R^d)^2$ the following system
\[
\mathfrak{A}_e=\begin{bmatrix}
-G_e&-D_e^*\\-D_e^*&-R_e
\end{bmatrix},\quad (H_ex_e)(\xi):=\mathcal{H}_e(\xi)x_e(\xi)=\begin{bmatrix}C_e(\xi)^{-1}&0\\0&L_e(\xi)^{-1}\end{bmatrix}\begin{bmatrix}
Q_e(\xi)\\ \phi_e(\xi)
\end{bmatrix}=\begin{bmatrix}U_e(\xi)\\ I_e(\xi)\end{bmatrix}
\]
and the skew-symmetric operator
\[
D_ex:=\tfrac{d}{d\xi}x ,\quad \dom D_e:=\{x_e\in H^1([0,\ell_e]) ~|~x_e(0)=x_e(\ell_e)=0\}. 
\]
The following boundary triplet is given by Lemma~\ref{lem:bt_TL}
\[
\Gamma_0^e(x_e):=\begin{bmatrix}
U_e(0)\\ U_e(\ell_e)
\end{bmatrix},\quad \Gamma_1^e(x_e):=\begin{bmatrix}
-\i I_e(0)\\ \i I_e(\ell_e)
\end{bmatrix}.
\]
This triplet fulfills for the operators $S_jH$ assumption (A5'), which was used in Section \ref{sec:interconnect} for Kirchhoff type interconnections. In particular, the boundary mappings are given by 
\begin{align*}
x_j=(x_0^j,x_1^j)\mapsto\begin{bmatrix}
\Gamma_{0,0}^eHx_0^e\\ \Gamma_{0,1}^eHx_0^e
\end{bmatrix}=\begin{bmatrix}
U_e(0)\\ U_e(\ell_e)
\end{bmatrix},\\ 
x_j=(x_0^j,x_1^j)\mapsto\begin{bmatrix} \Gamma_{1,0}^eHx_1^e\\ - \Gamma_{1,1}^eHx_1^j\end{bmatrix}=\begin{bmatrix}
-\i I_e(0)\\ \i I_e(\ell_e)
\end{bmatrix}.
\end{align*}

If we consider now the operators based on establishing continuity of $\Gamma_0^e$ at every vertex $v\in V$, then this leads to the node-type boundary triplet 
\[
\Gamma_0^V((x_e)_{e\in E})=(U(v))_{v\in V},\quad \Gamma_1^V((x_e)_{e\in E}):=\big\{-\i \sum_{e\sim v, (e,k)\sim v}{\rm sgn}(e,v)I_e(k\ell_e)\big\}_{v\in V},
\]
where the sum is taken over all edges $e$ which are adjacent to $v$ and $k=0,1$. Hence the voltage controlled network of transmission lines can be written as the following boundary control system 
\begin{align}
\label{eq:TL_1}
\dot z(t)=\mathfrak{A}Hz(t),\quad  u(t)=\Gamma_0^Vz(t),\quad y(t)=-\i\Gamma_1^Vz(t)
\end{align}
where 
\begin{align}
\label{eq:TL_2}
z=(z_e)_{e\in E},\quad \mathfrak{A}:=\begin{bmatrix}
-\bigoplus_{e\in E}G_e&\bigoplus_{e\in E}D_e^*\\\bigoplus_{e\in E}D_e^*&-\bigoplus_{e\in E}R_e
\end{bmatrix},\quad H:=\begin{bmatrix}\bigoplus_{e\in E}
C_e(\cdot)^{-1}&0\\0&\bigoplus_{e\in E}L_e(\cdot)^{-1}
\end{bmatrix}.
\end{align}
This can be viewed as having a voltage input at every node. Furthermore, the operator in (A3) is given by 
\[
A=\mathfrak{A}H|_{\ker \Gamma_0^V}=\begin{bmatrix}
-G&-D^*\\D&-R
\end{bmatrix}
\]
and has the property that $D$ is  the direct sum of derivatives with domain 
\[
\dom D=\{(U_e)_{e\in E}\in \oplus_{e\in E} H^1([0,\ell_e]),U_e(v)=0 \},\quad D(U_e)_{e\in E}=(\tfrac{d}{d\xi}U_e)_{e\in E}.
\]
Therefore, $D$ is injective and both operators $(I_X+D^*D)^{-1}$ and  $(I_X+DD^*)^{-1}$ are compact by the Sobolev embedding theorem. Hence, by Proposition~\ref{prop:compact},  $\mathfrak{A}H|_{\ker \Gamma_0^V}$ generates an exponentially stable semigroup. 
Furthermore, by Proposition~\ref{prop:pbe} the following power balance equation holds for all $u\in C^2([0,\infty),\C^{|V|})$
\begin{align*}
\sum_{e\in E}\re \langle\mathfrak{A}_eH_e z_e(t),H_e z_e(t)\rangle&=\re \langle u(t),y(t)\rangle-\sum_{e\in E}\left\langle\begin{bmatrix}-G_e&0\\0&-R_e
\end{bmatrix}H_e z_e(t),H_e z_e(t)\right\rangle,\\
\sum_{e\in E}(\langle H_ez_e(t),z_e(t)\rangle-\langle H_ez_e(0),z_e(0)\rangle)&\leq 2\int_0^t\re \langle u(\tau),y(\tau)\rangle d\tau-\sum_{e\in E}\left\langle\begin{bmatrix}G_e&0\\0&R_e
\end{bmatrix}H_e z_e(t),H_e z_e(t)\right\rangle.
\end{align*}

To study current controlled systems, we interchange the boundary triplet mappings which results in the following boundary triplet
\[
\hat \Gamma_0^V((x_e)_{e\in E})=\big\{\sum_{e\sim v, (e,k)\sim v}{\rm sgn}(e,v)I_e(k\ell_e)\big\}_{v\in V},\quad  ,\quad \hat \Gamma_1^V((x_e)_{e\in E}):=(-\i U(v))_{v\in V}.
\]
Here the resulting operator 
\[
\hat A:=\mathfrak{A}H|_{\ker \hat \Gamma_0^V}=\begin{bmatrix}
-G&-\hat D^*\\-\hat D&-R
\end{bmatrix}
\]
does in general not generate an exponentially stable semigroup, since $\hat D$
 has a one-dimensional kernel spanned by the function $c\in (\oplus_{e\in E}H^1([0,\ell_e])) ^2$ which is constant in the first $|E|$ entries which correspond to the voltages and zero in the remaining entries which correspond to the currents. Hence if $\ker G_e(\xi)$ does not contain the constant function for all $\xi$ from some non-empty interval and all edges $e\in E$, then $\begin{smallbmatrix}
 G\\\hat D
 \end{smallbmatrix}$ is injective and therefore Proposition \ref{prop:exp_stable} implies the exponential stability of the semigroup $\{\hat T(t)\}_{t\geq 0}$ generated by $\hat A$. Otherwise, $\ker \hat A=\Span\{c\}$ and Corollary~\ref{cor:hypocoercive} yields the existence of $M,\alpha>0$ such that the following holds:
 \[
\left\|\hat T(t)x_0-\frac{\langle x_0,c\rangle_{L^2}}{\|c\|_{L^2}}c\right\|_{L^2}\leq Me^{-\alpha t}\left\|x_0-\frac{\langle x_0,c\rangle_{L^2}}{\|c\|_{L^2}}c\right\|_{L^2}.
 \]
 
 Alternatively, for the current controlled case, we assume that $G$ is uniformly positive instead of $R$. Then one can use the loop based formulation which requires a set of fundamental circles $\mathcal{C}$ and constant currents $I(C)$ for all $C\in\mathcal{C}$
\begin{align}
\label{eq:current_control}
\Gamma_0^E((x_e)_{e\in E})=\big\{I(C)\big\}_{C\in \mathcal{C}},\quad  \quad \hat \Gamma_1^E((x_e)_{e\in E}):=\left(\i \sum_{e\in C, (k,l)\sim e} \sgn(e,C)U(k\ell_e)\right)_{C\in\mathcal{C}}.
\end{align}
However, in comparison to the voltage controlled case, we cannot directly conclude the injectivity of the resulting operator 
\[
\tilde A:=\mathfrak{A}H|_{\ker \hat \Gamma_0^E}=\begin{bmatrix}
-G&-\tilde D^*\\-\tilde D&-R
\end{bmatrix}.
\]
To achieve injectivity we need the additional assumption that a zero current on each fundamental cycle implies that the current is identically zero on every edge. Sufficient for this is that every edge is contained in a fundamental cycle. Furthermore, $\tilde A$ has a compact resolvent and hence, it generates an exponentially stable semigroup.

We consider now networks with $R_e=G_e=0$ for all $e\in E$ together with the following boundary conditions
\[
-\i \sum_{e\sim v, (e,k)\sim v}{\rm sgn}(e,v)I_e(k\ell_e)=-\i r_v U(v)
\]
for some $r_v\geq 0$ and all $v\in V$, which describe loads attached to the network nodes. Then for $L:=\diag (-\i r_v)_{v\in V}$ the matrices $-\i L=\diag (-r_v)_{v\in V}=\i L^*$ are dissipative. Hence, by Proposition~\ref{prop:node}, these boundary conditions lead to generators of contraction semigroups. This implies the generation of contraction semigroups also for all networks satisfying $R_e=G_e\geq 0$. 

\subsection*{Prosumer network example}
We conclude this section by applying our modeling approach to the small example network shown in Figure~\ref{fig:grid}. The network consists of five transmission lines $S_1,\ldots,S_5$ with lengths 
$\ell_1,\ldots,\ell_5>0$, four prosumers $P_1,\ldots, P_4$ which can either produce or consume power from the network. The underlying graph $G$ has four nodes $v_1,v_2,v_3,v_4$ and five edges $e_1,\ldots,e_5$ as shown in Figure~\ref{fig:grid}. 

The system is an AC three-phase system and the currents and voltages $I_i,U_i\in\R^3$ along the line $S_i$ with  $i=1,\ldots,5$ are modeled using the telegraph equations \eqref{eq:tele}. Furthermore, we do not consider a static load model but instead, we consider both the generators and the loads to act as (time-varying) inputs $u_{P_1},\ldots u_{P_4}$ to the system. Here we focus on the following two cases: either the generators and loads provide all voltage inputs or they provide all current inputs. In the voltage controlled case, the currents of the loads and the generators are considered formally as outputs $y_{P_1}, \ldots, y_{P_4}$ of the system, whereas in the current controlled case, the voltages are considered as outputs. The reasoning behind this is that the value of the output variables is already determined by the Kirchhoff current law and the Kirchhoff voltage law, respectively.

If we consider the voltage controlled case, then we equate the  voltages at the 
endpoints of transmission lines at the interconnection points. In the example network in Figure~\ref{fig:grid} we have four nodes which lead to the following four equations
\begin{align*}
U_{S_1}(t,0)=U_{S_4}(t,\ell_4)=U_{S_5}(t,0)=u_{P_1}(t), \\
U_{S_1}(t,\ell_1)=U_{S_2}(t,0)=u_{P_2}(t),\\
U_{S_2}(t,\ell_2)=U_{S_3}(t,0)=U_{S_5}(t,\ell_5)=u_{P_3}(t),\\
U_{S_3}(t,\ell_3)=U_{S_4}(t,0)=u_{P_4}(t).
\end{align*}
The common node voltages are used in the formulation of the boundary control system analogous to \eqref{eq:TL_1} and \eqref{eq:TL_2}
\[
\dot{z}=\mathfrak{A}Hz,\quad  u(t)=\begin{smallbmatrix} U_{S_1}(t,0)\\
    U_{S_2}(t,0)\\
    U_{S_3}(t,0)\\
    U_{S_4}(t,0) \end{smallbmatrix}, \quad y(t)=\begin{smallbmatrix}
        I_{S_1}(t,0)+I_{S_5}(t,0)-I_{S_4}(t,\ell_4)\\ I_{S_2}(t,0)-I_{S_1}(t,\ell_1)\\  I_{S_3}(t,0)-I_{S_2}(t,\ell_2)-I_{S_5}(t,\ell_5)\\ I_{S_4}(t,0)-I_{S_3}(t,\ell_3)
    \end{smallbmatrix},
\]
where the output $y(t)$ is the current of the prosumer and the output equation reflects Kirchhoffs current law at the prosumer nodes.

For the current controlled case, we consider the loop based formulation which is based on the spanning tree shown containing the edges $e_1,e_4$ and $e_5$ in Figure~\ref{fig:grid} and the induced fundamental cycles $C_1$ and $C_2$. Observe that the graph has the property that every edge is contained in a fundamental cycle. In particular, as we have discussed earlier in this section, we obtain the exponential stability of the underlying semigroup. The formulation of the current controlled system based on \eqref{eq:current_control} is straightforward and therefore omitted.
	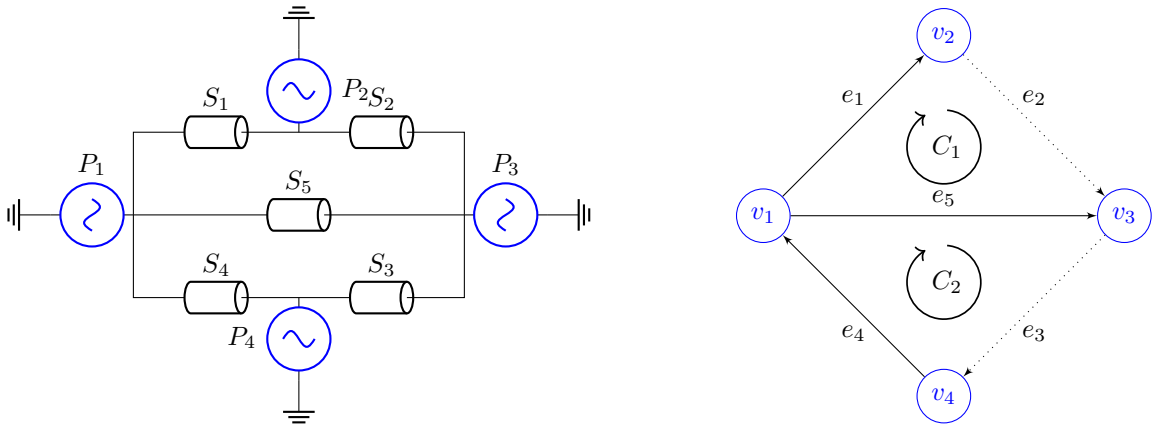
\begin{figure}[htpb!]
	         \centering
     \begin{subfigure}[b]{0.5\textwidth}
   \begin{center}
		\begin{circuitikz}[european,scale=0.55]\draw
			(1,0) node[ ground,rotate=270,label={[font=\normalsize]above:} ]{} (1,0);
			\draw[black](1,0) to[vsourcesin,color=blue,l =$P_1$](3,0);
			\draw[black](3,0)
			to[transmission line,color=black,l =$S_5$] (11,0); \draw[black] (11,0) to[vsourcesin,color=blue,l=$P_3$](13,0); \draw[black] (13,0) node[ground,rotate=90,label={[font=\normalsize]above:}]{} (14,0);
			\draw[black] (3,0) -- (3,2)
			to[transmission line, l=$S_1$,color=black] (7,2)
			to[transmission line,l=$S_2$,color=black] (11,2)
			-- (11,0);
			\draw[black]
			(3,0)-- (3,-2)
			to[transmission line, l=$S_4$,color=black] (7,-2)
			to[transmission line, l=$S_3$,color=black] (11,-2)
			--(11,0)	;
			\draw (7,4) node[ground,rotate=180,label={[font=\normalsize]right:}]{} (7,4);
			\draw[black] (7,4)
			to[vsourcesin,color=blue,l=$P_2$](7,2);
			\draw (7,-4) node[ground,label={[font=\normalsize]right:}]{}(7,-4);
			\draw[black](7,-4)
			to[vsourcesin,color=blue, l=$P_4$](7,-2);
		\end{circuitikz}
	\end{center}
     \end{subfigure}
     \hfill
     \begin{subfigure}[b]{0.4\textwidth}
   \begin{center}
		\begin{tikzpicture}[scale= 0.6]
	
	\tikzset{vertex/.style = {shape=circle,draw,minimum size=1.5em}}
	\tikzset{edge/.style = {->,> = latex'}}
	% vertices
	\node[vertex,blue] (b) at  (0,0.25) {$v_1$};
	\node[vertex,blue] (c) at  (8,0.25) {$v_3$};
	\node[vertex,blue] (a1) at (4,4.25) {$v_2$};
	\node[vertex,blue] (a2) at (4,-3.75) {$v_4$};
	\node[] (c1) at (4,1.75) {\CricArrowRight{$C_1$}};
	\node[] (c2) at (4,-1.25) {\CricArrowRight{$C_2$}};
	%edges
	\draw[edge] (b) -- (a1) node[midway,above=4pt]{$e_1$};
	\draw[edge] (a2) -- (b) node[midway,below= 4pt]{$e_4$};
	\draw[edge] (b) -- (c) node[midway,above]{$e_5$};
	\draw[edge,dotted] (c) -- (a2) node[midway,below= 4pt]{$e_3$};
	\draw[edge,dotted] (a1) -- (c) node[midway,above=4pt]{$e_2$};

	\end{tikzpicture}
	\end{center}
     \end{subfigure}
    \caption{The underlying graph of the transmission network consists of four nodes $v_1,\ldots,v_4$ and five edges  $e_1,\ldots, e_5$. Furthermore, we have a spanning tree that consists of the edges $e_1,e_5$, and $e_4$ and two fundamental cycles $C_1$ and $C_2$.}
    \label{fig:grid}
	\end{figure}

\section*{Acknowledgement}
We would like to thank Mikael Kurula, Friedrich Philipp, Jonathan Rohleder and Nathanael Skrepek for their valuable discussions on this draft. This work was funded by the Deutsche Forschungs Gemeinschaft within the Priority Programme 1984 and the Collaborative Research Center 910: Control of self-organizing nonlinear systems. Furthermore, HG acknowledges the funding from the Wenner Gren Foundation. DH has been supported by Bundesministerium für Bildung und Forschung (BMBF)  EKSSE: Energieeffiziente Koordination und Steuerung des Schienenverkehrs in Echtzeit (grant no. 05M22KTB).

\bibliographystyle{alpha}

\bibliography{references}

\newcommand{\etalchar}[1]{$^{#1}$}
\begin{thebibliography}{LGRW{\etalchar{+}}22}

\bibitem[AJ14]{AugJ14}
B.~Augner and B.~Jacob.
\newblock Stability and stabilization of infinite-dimensional linear
  port-{H}amiltonian systems.
\newblock {\em Evolution Equations \& Control Theory}, 3(2):207, 2014.

\bibitem[AM13]{AalM13}
A.~Aalto and J.~Malinen.
\newblock Compositions of passive boundary control systems.
\newblock {\em Mathematical Control and Related Fields}, 3(1):1--19, 2013.

\bibitem[Aug20]{Aug20}
B.~Augner.
\newblock Well-posedness and stability for interconnection structures of
  port-{H}amiltonian type.
\newblock In J.~Kerner, H.~Laasri, and D.~Mugnolo, editors, {\em Control Theory
  of Infinite-Dimensional Systems}, pages 1--52, Cham, 2020. Springer
  International Publishing.

\bibitem[BB21]{BanB21}
J.~Banasiak and A.~B{\l}och.
\newblock Telegraph systems on networks and port-{H}amiltonians. i. {B}oundary
  conditions and well-posedness.
\newblock {\em arXiv:2102.07481}, 2021.

\bibitem[BHdS20]{Behr20}
J.~Behrndt, S.~Hassi, and H.~de~Snoo.
\newblock {\em Boundary Value Problems, {W}eyl functions, and Differential
  Operators}.
\newblock Springer Nature, 2020.

\bibitem[DM91]{Derk91}
V.~Derkach and M.M. Malamud.
\newblock Generalized resolvents and the boundary value problems for
  {H}ermitian operators with gaps.
\newblock {\em J.\ Funct.\ Anal.}, 95(1):1--95, 1991.

\bibitem[EK18]{EggK18}
H.~Egger and T.~Kugler.
\newblock Damped wave systems on networks: exponential stability and uniform
  approximations.
\newblock {\em Numerische Mathematik}, 138:839--867, 2018.

\bibitem[EKLS{\etalchar{+}}18]{EggKLSMM18}
H.~Egger, T.~Kugler, B.~Liljegren-Sailer, N.~Marheineke, and V.~Mehrmann.
\newblock On structure-preserving model reduction for damped wave propagation
  in transport networks.
\newblock {\em SIAM J.\ Sci.\ Comput.}, 40(1):A331--A365, 2018.

\bibitem[EN00]{EngN00}
K.-J. Engel and R.~Nagel.
\newblock {\em One-parameter semigroups for linear evolution equations}.
\newblock Springer, 2000.

\bibitem[GGK89]{Gorb89}
V.I. Gorbachuk, M.L. Gorbachuk, and A.N. Kochubei.
\newblock Extension theory for symmetric operators and boundary value problems
  for differential equations.
\newblock {\em Ukrainian Mathematical Journal}, 41(10):1117--1129, 1989.

\bibitem[GHRvdS21]{GerHRS21}
H.~Gernandt, F.E. Haller, T.~Reis, and A.J. van~der Schaft.
\newblock Port-hamiltonian formulation of nonlinear electrical circuits.
\newblock {\em J.\ Geom.\ Phys.}, 159:103959, 2021.

\bibitem[GT21]{GernTrun21}
H.~Gernandt and C.~Trunk.
\newblock Locally finite extensions and {G}esztesy--{\v{s}}eba realizations for
  the {D}irac operator on a metric graph.
\newblock In {\em Operator Theory: Proceedings of the International Conference
  on Operator Theory, Hammamet, Tunisia, April 30-May 3, 2018}, page~25. Walter
  de Gruyter GmbH \& Co KG, 2021.

\bibitem[HRB75]{HoRB75}
C.-W. Ho, A.~Ruehli, and P.~Brennan.
\newblock The modified nodal approach to network analysis.
\newblock {\em IEEE Trans.\ Circuits Syst.}, 22:504--509, 1975.

\bibitem[JZ12]{JacZ12}
B.~Jacob and H.~Zwart.
\newblock {\em Linear port-{H}amiltonian systems on infinite-dimensional
  spaces}.
\newblock Operator Theory: Advances and Applications, 223.
  Birkh{\"a}user/Springer Basel AG, Basel CH, 2012.

\bibitem[Kat00]{Kat95}
T.~Kato.
\newblock {\em Perturbation Theory for Linear Operators}.
\newblock Springer, Berlin, 2000.

\bibitem[KZ15]{KurZ15}
M.~Kurula and H.~Zwart.
\newblock Linear wave systems on n-{D} spatial domains.
\newblock {\em Internat.\ J.\ Control}, 88(5):1063--1077, 2015.

\bibitem[KZvdSB10]{KurZVB10}
M.~Kurula, H.~Zwart, A.~van~der Schaft, and J.~Behrndt.
\newblock Dirac structures and their composition on {H}ilbert spaces.
\newblock {\em J.\ Math.\ Anal.\ Appl.}, 372(2):402--422, 2010.

\bibitem[LGRW{\etalchar{+}}22]{LeGRWLM2022}
Y.~Le~Gorrec, H.~Ramirez, Y.~Wu, N.~Liu, and A.~Macchelli.
\newblock {\em Energy Shaping Control of 1D Distributed Parameter Systems},
  pages 3--26.
\newblock Springer International Publishing, Cham, 2022.

\bibitem[LGZM05]{LeGZM05}
Y.~Le~Gorrec, H.~Zwart, and B.~Maschke.
\newblock Dirac structures and boundary control systems associated with
  skew-symmetric differential operators.
\newblock {\em SIAM J.\ Control Optim.}, 44, 2005.

\bibitem[MS06]{MalS06}
J.~Malinen and O.J. Staffans.
\newblock Conservative boundary control systems.
\newblock {\em J.\ Differential Equations}, 231, dec 2006.

\bibitem[MS07]{MalS07}
J.~Malinen and O.J. Staffans.
\newblock Impedance passive and conservative boundary control systems.
\newblock {\em Complex Anal. Oper. Theory}, 1, 2007.

\bibitem[MU23]{UnM22}
Volker Mehrmann and Benjamin Unger.
\newblock Control of port-{H}amiltonian differential-algebraic systems and
  applications.
\newblock {\em Acta Numerica}, 32:395–515, 2023.

\bibitem[NPS22]{NedPS22}
N.~Nedialkov, J.~D. Pryce, and L.~Scholz.
\newblock An energy-based, always index $\leq$ 1 and structurally amenable
  electrical circuit model.
\newblock {\em SIAM J.\ Sci.\ Comput.}, 44(4):B1122--B1147, 2022.

\bibitem[PSF{\etalchar{+}}21]{PhiSFMW21}
F.~Philipp, M.~Schaller, T.~Faulwasser, B.~Maschke, and K.~Worthmann.
\newblock Minimizing the energy supply of infinite-dimensional linear
  port-{H}amiltonian systems.
\newblock {\em IFAC-PapersOnLine}, 54(19):155--160, 2021.
\newblock 7th IFAC Workshop on Lagrangian and Hamiltonian Methods for Nonlinear
  Control LHMNC 2021.

\bibitem[RCvdSS20]{RasCSS20}
R.~Rashad, F.~Califano, A.~J. van~der Schaft, and S.~Stramigioli.
\newblock Twenty years of distributed port-{H}amiltonian systems: a literature
  review.
\newblock {\em IMA J.\ Math.\ Control Inform.}, 37(4):1400--1422, 2020.

\bibitem[Rei14]{Rei14}
T.~Reis.
\newblock Mathematical modeling and analysis of nonlinear time-invariant {RLC}
  circuits.
\newblock In P.~Benner, R.~Findeisen, D.~Flockerzi, U.~Reichl, and
  K.~Sundmacher, editors, {\em Large-Scale Networks in Engineering and Life
  Sciences}, Modeling and Simulation in Science, Engineering and Technology,
  pages 125--198. Birkh\"auser, Basel, 2014.

\bibitem[RLGMZ14]{RamLGMZ14}
H.~Ramirez, Y.~Le~Gorrec, A.~Macchelli, and H.~Zwart.
\newblock Exponential stabilization of boundary controlled port-{H}amiltonian
  systems with dynamic feedback.
\newblock {\em IEEE Trans. Autom. Control}, 59(10):2849--2855, 2014.

\bibitem[RPS23]{ReiPS23}
T.~Reis, F.~Philipp, and M.~Schaller.
\newblock Infinite-dimensional port-{H}amiltonian systems -- a system node
  approach.
\newblock {\em arXiv:2302.05168}, 2023.

\bibitem[RT05]{ReiT05}
T.~Reis and C.~Tischendorf.
\newblock Frequency domain methods and decoupling of linear infinite
  dimensional differential algebraic systems.
\newblock {\em J.\ Evol.\ Equ.}, 5(3):357--385, 2005.

\bibitem[Sch12]{Sch12}
K.~Schm\"{u}dgen.
\newblock {\em Unbounded Self-adjoint Operators on Hilbert Space}.
\newblock Springer, Berlin, 2012.

\bibitem[Skr21]{Skr21}
N.~Skrepek.
\newblock {\em Linear port-Hamiltonian Systems on Multidimensional Spatial
  Domains}.
\newblock PhD thesis, Bergische Universität Wuppertal, 2021.

\bibitem[Sta02]{Sta02}
O.J. Staffans.
\newblock Passive and conservative continuous-time impedance and scattering
  systems. part i: Well-posed systems.
\newblock {\em Math. Control Signals Systems}, 15:291--315, 2002.

\bibitem[Sta05]{Sta05}
O.J. Staffans.
\newblock {\em Well-Posed Linear Systems}.
\newblock Encyclopedia of Mathematics and its Applications. Cambridge
  University Press, 2005.

\bibitem[TW22]{TroW22}
S.~Trostorff and M.~Waurick.
\newblock Characterisation for exponential stability of port-{H}amiltonian
  systems.
\newblock {\em arXiv:2201.10367}, 2022.

\bibitem[Vil07]{Vill07}
J.A. Villegas.
\newblock {\em A Port-{H}amiltonian Approach to Distributed Parameter Systems}.
\newblock PhD thesis, University of Twente, Netherlands, May 2007.

\bibitem[VZLG09]{VilZLGM09}
J.A. Villegas, H.~Zwart, and B.~Le~Gorrec, Y.and~Maschke.
\newblock Exponential stability of a class of boundary control systems.
\newblock {\em IEEE Trans.\ Automat.\ Control}, 54, jan 2009.

\bibitem[WW20]{SAW20}
M.~Waurick and S.~Wegner.
\newblock Dissipative extensions and port-{H}amiltonian operators on networks.
\newblock {\em J.\ Differential Equations}, 269(9):6830--6874, 2020.

\bibitem[WZ22]{WauZ22}
M.~Waurick and H.~Zwart.
\newblock Asymptotic stability of port-{H}amiltonian systems.
\newblock {\em arXiv:2210.11775}, 2022.

\bibitem[ZLGM16]{ZwaLGM16}
H.~Zwart, Y.~Le~Gorrec, and B.~Maschke.
\newblock Building systems from simple hyperbolic ones.
\newblock {\em Systems Control Lett.}, 91, 2016.

\bibitem[ZLGMV09]{ZwaLGMV09}
H.~Zwart, Y.~Le~Gorrec, B.~Maschke, and J.A. Villegas.
\newblock Well-posedness and regularity of hyperbolic boundary control systems
  on a one-dimensional spatial domain.
\newblock {\em ESAIM Control Optimisation and Calculus of Variations}, 16,
  2009.

\end{thebibliography}

\end{document}